\documentclass[11pt]{article}%
\usepackage{amssymb,amsmath,amsfonts,amsthm,array,bm,bbm,color}%
\usepackage{graphicx}
\providecommand{\U}[1]{\protect\rule{.1in}{.1in}}

\numberwithin{equation}{section}

\newtheorem{theorem}{Theorem}[section]
\newtheorem{lemma}[theorem]{Lemma}
\newtheorem{corollary}[theorem]{Corollary}
\newtheorem{proposition}[theorem]{Proposition}
\newtheorem{remark}[theorem]{Remark}
\newtheorem{example}[theorem]{Example}

\def\<{\langle}
\def\>{\rangle}
\def\E{\mathbb{E}}
\def\P{\mathbb{P}}
\def\R{\mathbb{R}}
\def\T{\mathbb{T}}
\def\Z{\mathbb{Z}}

\def\div{{\rm div}}

\begin{document}

\title{A scaling limit for Vlasov equations with electrostatic fluctuations}

\author{Franco Flandoli\thanks{Email: franco.flandoli@sns.it. Scuola Normale Superiore of Pisa, Piazza dei Cavalieri 7, 56124 Pisa, Italy}
\qquad Dejun Luo\thanks{Email: luodj@amss.ac.cn. SKLMS, Academy of Mathematics and Systems Science, Chinese Academy of Sciences, Beijing 100190, China and School of Mathematical Sciences, University of Chinese Academy of
Sciences, Beijing 100049, China}}
\maketitle

\vspace{-15pt}

\begin{abstract}
We consider a Vlasov equation for a plasma with a given constant magnetic field, and introduce a white noise perturbation of the electric field in the electrostatic approximation, with a discussion of the motivations of such random perturbation. We prove that diffusion in velocity emerges in a suitable scaling limit of the noise, and also discuss the physical relevance of this result.
\end{abstract}


\section{Introduction}

Consider the Vlasov-type equation
  $$\partial_t f + v\cdot\nabla_x f+ (E_\rho + B v\times \bm{e}_3) \cdot\nabla_v f = 0,$$
where $f=f(x,v,t)$ is the density of plasma particles in the phase space $\T^3\times \R^3$, $\rho(x,t)= \int_{\R^3} f(x,v,t) \, dv$; $B\in \R$ is a constant, $\bm{e}_3= (0,0,1)^\ast$ and $v\times \bm{e}_3= (v_2,-v_1, 0)^\ast$; moreover,
  $$E_{\rho(t)} (x) = \int_{\T^3} \nabla G(x-y) \rho(y,t)\, dy, $$
$G$ being the Green function on $\T^3$, $G(x)\approx -C |x|^{-1}$ for $x$ near the origin. $E_\rho$ stands for the electric field associated to the particle density, while $B \bm{e}_3$ represents some fixed external magnetic field.

Here we briefly discuss how to introduce noise in the above equation, see Section \ref{sec-particular-noise} for a more detailed physical explanation. Assume that $f= \bar f + \tilde f$, where $\bar f$ (resp. $\tilde f$) represents the large (resp. small) scale part of the density; accordingly, we have $\rho= \int \bar f\,d v + \int \tilde f\,d v=: \bar \rho + \tilde \rho$, and the equation for $\bar f$ becomes
  $$\partial_t \bar f + v\cdot\nabla_x \bar f+ (E_{\bar \rho} + B v\times \bm{e}_3) \cdot\nabla_v \bar f + E_{\tilde \rho}\cdot\nabla_v \bar f = 0. $$
Motivated by stochastic model reduction, we approximate the electric field $E_{\tilde \rho}$ corresponding to small scale density by some Gaussian random field
  $$W= W(x,t) = \sum_{k} \sigma_k(x) W^k_t, $$
where $\{\sigma_k \}_k$ are some smooth gradient vector fields on $\T^3$ and $\{W^k \}_k$ are independent standard Brownian motions. We stress a major difference from recent works mentioned below, where the noise is often assumed to be divergence-free in the space variable, while here it is of gradient type since it arises from small scale fluctuations of electric field; a key point is that the noise is divergence-free in the velocity variable.

We can write the above equation as a stochastic partial differential equation (SPDE) with transport noise:
  $$d \bar f + v\cdot\nabla_x \bar f\, dt + (E_{\bar \rho} + B v\times \bm{e}_3) \cdot\nabla_v \bar f \, dt + \sum_k \sigma_k \cdot\nabla_v \bar f \circ d W^k_t =0.$$
Here, $\circ d$ means the Stratonovich stochastic differential, a choice which is motivated by the Wong--Zakai principle (cf. \cite{BF95}). This is the equation treated below and we shall omit the bar in $ \bar f$ to simplify notation.

Technically speaking, we have been motivated, to introduce the above noise and
investigate the scaling limit, on one side by a series of recent works on a
new scaling limit based on the It\^{o}-Stratonovich duality (see e.g.
\cite{FGL21a, FGL22, Gal, Luo21} and more references at the beginning of Section \ref{subsec-scaling-limit}), and on the other side by the work
\cite{BardosBesse} which contains two very interesting scaling limit results
related to turbulent perturbations. Our approach and assumptions are
technically very different and therefore complement the study of
\cite{BardosBesse} by a different viewpoint. From the physical side, moreover,
we have been strongly motivated by the increasing knowledge and understanding
of turbulence in plasma. As an example of works discussing topics and ideas
which may connect with our noise model and limit result, let us mention
\cite{Hutch1, Hutch2} which insist on coherent dynamical structures;
our noise is composed of random structures mimicking that intuition. We also mention
\cite{Ewart et al, Nastac et al} which discuss in particular the role
that turbulence may have in restoring a form of chaoticity typical of
molecular collisions, in a framework like non-collisional plasma that a priori
could then miss any statistical mechanics approach. In Subsection \ref{subsect collisions} we shortly discuss this issue in more detail.

\subsection{Weak existence of stochastic Vlasov equations}

We will write the time variable $t$ as a subscript to save space. As discussed in the last part, we consider the stochastic Vlasov equation
  \begin{equation}\label{stoch-Vlasov}
  d f + v\cdot\nabla_x f\, dt + (E_{\rho} + B v\times \bm{e}_3) \cdot\nabla_v f \, dt + \sum_k \sigma_k \cdot\nabla_v f \circ d W^k_t =0,
  \end{equation}
where $\rho_t(x)= \int f_t(x,v)\, dv$ and
  $$E_{\rho_t}(x) = (\nabla G\ast \rho_t)(x)= \int_{\T^3} \nabla G(x-y) \rho_t(y)\, dy. $$

For simplicity, we assume the noise $W(x,t)= \sum_k \sigma_k(x) W^k_t$ in \eqref{stoch-Vlasov} is translation-invariant and smooth in the space variable, namely, the covariance function
  $$Q(x-y)= \sum_k \sigma_k(x)\otimes \sigma_k(y)$$
satisfies $Q\in C_b^\infty(\T^3)$. Moreover, we can assume that
  \begin{equation}\label{covariance}
  Q(0)= \sum_k \sigma_k(x)\otimes \sigma_k(x) = 2\kappa I_3
  \end{equation}
for some $\kappa>0$, $I_3$ being the $3\times 3$ unit matrix; see Example \ref{exa-noises} below for mathematical construction of such noises. Under these assumptions, equation \eqref{stoch-Vlasov} can be written in It\^o form as
  \begin{equation}\label{stoch-Vlasov-Ito}
  d f + v\cdot\nabla_x f\, dt + (E_{\rho} + B v\times \bm{e}_3) \cdot\nabla_v f \, dt + \sum_k \sigma_k \cdot\nabla_v f \, d W^k_t = \kappa \Delta_v f\, dt,
  \end{equation}
where $\kappa \Delta_v$ is the It\^o-Stratonovich corrector; we remark that this operator is a ``fake'' dissipation term since it is exactly cancelled by the martingale part in energy estimates.  We endow the above equation with an initial condition $f_0\in L^1(\T^3\times \R^3,\R_+)$, together with further regularity to be specified later. The meaning of weak solution to \eqref{stoch-Vlasov-Ito} is defined as usual against test functions $\phi\in C_c^\infty(\T^3\times \R^3)$. The main result of this part is the existence of weak solutions to \eqref{stoch-Vlasov-Ito} which will be proved in Section \ref{sec-proof-existence}.

\begin{theorem}[Existence]\label{thm-existence}
Assume that the initial condition $f_0\in (L^1\cap L^3)(\T^3\times \R^3,\R_+)$ has finite kinetic energy
  $$\mathcal K(f_0):= \int |v|^2 f_0(x,v)\, dx dv<+\infty.$$
Then the stochastic Vlasov equation \eqref{stoch-Vlasov-Ito} admits a weak solution $f\in L^\infty\big(0,T; (L^1\cap L^3)(\T^3\times \R^3,\R_+) \big)$ such that
\begin{itemize}
\item[\rm(a)] for any $p\in [1,3]$, $\P$-a.s. $\|f_t \|_{L^p_{x,v}} \le \|f_0 \|_{L^p_{x,v}}$ for all $t\in [0,T]$;
\item[\rm(b)] for all $q\ge 1$, it holds $\E \mathcal K(f_t)^{2q} \le C_{q,T}<\infty$  for all $t\in [0,T]$.
\end{itemize}
\end{theorem}

The $L^3$-integrability of $f_0$ looks a little strange; it implies the existence of a sequence of approximate solutions bounded in $ L^3(\T^3\times \R^3)$, which will be used in \eqref{eq-L-3} in the proof of Lemma \ref{lem-Holder} below.

\subsection{Scaling limit for stochastic Vlasov equations}\label{subsec-scaling-limit}

In recent years, there have been intensive studies on the It\^o-Stratonovich diffusion limit, starting from the work \cite{Gal} which shows that, under a suitable rescaling of the space covariance of noises, a sequence of stochastic linear transport equations on the torus converge to the deterministic heat equation. This result was soon extended in \cite{FGL21a} to the vorticity form of stochastic 2D Euler equation with transport noise, where the limit equation becomes the deterministic 2D Navier-Stokes equation; see also \cite{Luo21} for similar result on the stochastic 2D inviscid Boussinesq system. The paper \cite{FGL22} extends the results of \cite{Gal} to stochastic transport-diffusion equations on bounded domains with Dirichlet boundary condition, while \cite{Hua25} and \cite{LXZ25} deal with the cases of compact manifold and full space $\R^d$, respectively. We refer to \cite{CL23, Agr24a, QS24, LiuLuo24} for works treating various models in different setting, and to \cite{FGL24, ZhaHua24} for quantitative estimates on convergence rates in the scaling limits. We also mention the recent papers \cite{FlaLuo24, BFL24, BL24} dealing with some 3D fluid models where a stretching part of noise is involved. See the lecture notes \cite{FL23} for more comprehensive introductions to this field.

Motivated by the above works, in this part we show that, under a similar scaling of the noise in \eqref{stoch-Vlasov-Ito}, the stochastic Vlasov equations converge weakly to the deterministic Vlasov equation with an extra dissipation term:
  \begin{equation}\label{limit-Vlasov}
  \partial_t \bar f + v\cdot\nabla_x \bar f+ (E_{\bar\rho} + B v\times \bm{e}_3) \cdot\nabla_v \bar f = \kappa\Delta_v \bar f, \quad \bar f|_{t=0}= f_0,
  \end{equation}
where $\bar\rho_t (x)= \int \bar f_t(x,v)\, dv$.

To this end, assume we are given a sequence of smooth noise
  $$W_N= W_N(x,t) = \sum_k \sigma_k^N(x) W^k_t, $$
with covariance functions
  \begin{equation}\label{covariance-N}
  Q_N(x-y)= \sum_k \sigma_k^N(x)\otimes \sigma_k^N(y), \quad x,y\in \T^3
  \end{equation}
satisfying
  \begin{equation}\label{condition-covariance}
  Q_N(0)= 2\kappa I_3\ (\forall\, N\ge 1) \quad \mbox{and} \quad \lim_{N\to \infty} \|Q_N \|_{L^r}= 0
  \end{equation}
for all $r\in [1,\infty)$. We will provide in Example \ref{exa-noises} an explicit choice of noises with the above properties, see also Section \ref{sec-particular-noise} for more physical discussions. Note that we can regard $Q_N$ as an operator acting on vector fields $A:\T^3\to \R^3$ defined by
  $$(Q_N A)(x)= (Q_N\ast A)(x)= \int_{\T^3} Q_N(x-y) A(y)\, dy, \quad x\in \T^3. $$

Here is the second main result of this paper.

\begin{theorem}[Scaling limit]\label{thm-limit}
Assume that the initial condition $f_0$ belongs to $(L^1\cap L^3)(\T^3\times \R^3)$ with finite kinetic energy
  $$\mathcal K(f_0)= \int_{\T^3\times \R^3} |v|^2 f_0(x,v)\, dx dv<+\infty. $$
Let $W_N$ be a sequence of smooth noises whose covariance functions $Q_N$ verify the condition \eqref{condition-covariance}. For any $N\ge 1$, let $f^N$ be a weak solution to the stochastic Vlasov equation
  \begin{equation}\label{stoch-Vlasov-Ito-N}
  d f^N + v\cdot\nabla_x f^N\, dt + (E_{\rho^N} + B v\times \bm{e}_3) \cdot\nabla_v f^N \, dt + \sum_k \sigma^N_k \cdot\nabla_v f^N \, d W^k_t = \kappa \Delta_v f^N\, dt
  \end{equation}
with initial condition $f^N_0= f_0$. Then any convergent subsequence of $\{f^N\}_N$ converge weakly, in the topology of $C\big([0,T], H^{-\epsilon }_{x,v, loc} \big)$, to weak solutions of \eqref{limit-Vlasov}.
\end{theorem}

We remark that, by Theorem \ref{thm-existence}, the approximate solutions $\{f^N\}_N$ satisfy the properties (a) and (b) therein; hence, the weak limits of $\{f^N\}_N$ belong to the same regularity class. However, we are unable to prove uniqueness of solutions to \eqref{limit-Vlasov} under such regularity conditions. Therefore, we cannot claim that the whole sequence $\{f^N\}_N$ is weakly convergent.

This paper is organized as follows. After some preparations in Section 2, we provide in Section \ref{sec-proof-existence} the proof of Theorem \ref{thm-existence} which shows the existence of weak solutions to the stochastic Vlasov equation \eqref{stoch-Vlasov-Ito}.
Then we prove the above scaling limit result (Theorem \ref{thm-limit}) in Section \ref{sec-proof-limit}. Since the idea is very close to the existence proof, we omit the details and only show why the martingale part vanishes in the limit. We provide in Section \ref{sec-particular-noise} detailed discussions on the choice of physically well motivated noises satisfying the condition \eqref{condition-covariance}.

\section{Useful lemmas and example of noises}

In the proof below we will use some classical lemmas which can be found in \cite{Horst82, HH84}, where the results are proved on the phase space $\R^N \times \R^N$ with $N\ge 2$. In the sequel, $\omega_N$ is the surface area of $N-1$ dimensional sphere, and $\lambda^N$ is the Lebesgue measure on $\R^N$.

The first result is taken from \cite[Lemma 3.2]{HH84}, see also \cite[Lemma 5.5]{Horst82}.

\begin{lemma}\label{app-lem-density}
Assume $1<p\le \infty$, $f\in L^p(\R^{2N})$, $f\ge 0$, $\mathcal K(f):= \int_{\R^{2N}} |v|^2 f(x,v)\, dx dv<\infty$. Let $\rho(x)= \int f(x,v)\, dv$ and $r= r(p):= (2p +N(p-1))/(2+ N(p-1))$ if $p<\infty$, and $r=r(p)= (2+N)/N$ if $p=\infty$. Then $\rho\in L^r(\R^N)$ and
  $$\|\rho \|_{L^r} \le C \|f \|_{L^p}^{2p'/(N+2p')} \mathcal K(f)^{N/(N+2p')} $$
with $C= 1+(\omega_N/N)^{1/p'}$.
\end{lemma}

The next estimate can be proved by H\"older's inequality.

\begin{lemma}\label{app-lem-compact}
Let $p\in[1, \infty]$, $f\in L^p(\R^{2N})$, $K\subset \R^N$ compact, $\rho_K(x):= \int_{K} f(x,v)\, dv$, $x\in \R^N$. Then $\rho_K\in L^p(\R^N)$ and $\|\rho_K \|_{L^p} \le (\lambda^N(K))^{1/p'} \|f \|_{L^p}$.
\end{lemma}

The following estimates can be found in \cite[Lemma 3.5]{HH84}.

\begin{lemma}\label{app-lem-product}
Let $p\in (1, \infty]$, $f_1, f_2\in L^p(\R^{2N})$, $a\in [0,N)$.

{\rm (i)} If $p<\infty$, $\frac1p + \frac1{r(p)} + \frac{a}N =2 $ and $\mathcal K(|f_2|) <+\infty$, then
  $$\aligned
  & \int_K \int_{\R^{2N}} \frac1{|x-y|^a} |f_1(x,v)|\, |f_2(y,w)| \, dx dv dy dw\\
  &\le C (\lambda^{2N} (K))^{1/p'} \|f_1 \|_{L^p} \|f_2 \|_{L^p}^{2p'/(N+2p')} \mathcal K(|f_2|)^{N/(N+2p')}
  \endaligned $$
for all compact sets $K\subset \R^{2N}$.

{\rm (ii)} If $ \frac2{r(p)} + \frac{a}N =2 $ and $\mathcal K(|f_1|) \vee \mathcal K(|f_2|) <+\infty$, then
  $$\aligned
  & \int_{\R^{2N}} \int_{\R^{2N}} \frac1{|x-y|^a} |f_1(x,v)|\, |f_2(y,w)| \, dx dv dy dw\\
  &\le C \big[\|f_1 \|_{L^p} \|f_2 \|_{L^p} \big]^{2p'/(N+2p')} \big[\mathcal K(|f_1|) \mathcal K(|f_2|) \big]^{N/(N+2p')} .
  \endaligned $$
\end{lemma}

We finish this section by giving an example of noises verifying the properties \eqref{condition-covariance}.

\begin{example}\label{exa-noises}
Let $\Z^3_0:= \Z^3\setminus \{0\}$ be the set of nonzero lattice points, with the partition $\Z^3_0= \Z^3_+ \cup \Z^3_-$ satisfying $\Z^3_+ =- \Z^3_-$. Set $a_k= \frac{k}{|k|}$ for $k\in \Z^3_+$ and $a_k= -\frac{k}{|k|}$ for $k\in \Z^3_-$. Now we define the gradient vector fields on $\T^3$ by $\sigma_k(x)= a_k e^{2\pi ik\cdot x}$, $x\in \T^3,\, k\in \Z^3_0$; the family $\{\sigma_k\}_{k\in \Z^3_0}$ constitutes a complete orthonormal system of the space of square integrable and gradient fields on $\T^3$ with zero mean. Let
  $$W_N(x,t) = \sqrt{3\kappa} \sum_{k} \Gamma^N_k \sigma_{k}(x) W^{k}_t, $$
where the coefficients $\Gamma^N= (\Gamma^N_k)_k \in \ell^2(\Z^3_0)$ are radial symmetric in $k\in \Z^3_0$ and satisfy $\|\Gamma^N \|_{\ell^2}=1$, $\{W^{k} \}_{k\in \Z^3_0}$ is a family of complex Brownian motions with quadratic variation
  $$\big[W^{k}, W^{l}\big]_t = 2 \delta_{k,-l}\, t. $$
Then, for any $x,y\in \T^3$, the covariance function verifies
  $$\aligned
  Q_N(x-y) &= 6\kappa\sum_{k} \big(\Gamma^N_k \big)^2 \sigma_{k}(x)\otimes \sigma_{-k}(y)
  = 6\kappa\sum_{k} \big(\Gamma^N_k \big)^2 \frac{k\otimes k}{|k|^2} e^{2\pi ik\cdot (x-y)}.
  \endaligned $$

As in \cite{FlaLuo21, Gal}, we can take $\{\Gamma^N \}_N \subset \ell^2(\Z^3_0)$ such that
  $$\big\| \Gamma^N \big\|_{\ell^2}=1\quad \mbox{for all } N\ge 1, \mbox{ but} \quad \big\|\Gamma^N \big\|_{\ell^\infty} \to 0 \quad \mbox{as } N\to \infty.$$
Now it is not difficult to show that $Q_N(0)=2\kappa I_3$, cf. \cite[(2.3)]{FlaLuo21}; moreover,
  $$ |Q_N(x)| \le 6 \kappa\sum_{k} \big(\Gamma^N_k \big)^2 = 6\kappa \quad \mbox{for all } x\in \T^3, $$
and
  $$ \|Q_N \|_{L^2} = 6\kappa \bigg[\sum_{k} \big(\Gamma^N_k \big)^4 \Big| \frac{k\otimes k}{|k|^2}\Big|^2 \bigg]^{1/2} \le 6\kappa\, \big\|\Gamma^N \big\|_{\ell^\infty}.$$
Fix any $r\in [1,\infty)$; if $r\in [1,2]$, then
  $$ \|Q_N \|_{L^r} \le \|Q_N \|_{L^2} \le 6\kappa\, \big\|\Gamma^N \big\|_{\ell^\infty}, $$
while if $r\in (2,\infty)$, then by interpolation with parameters $\frac1r = \frac{2/r}2 + \frac{1-2/r}\infty$,
  \[ \|Q_N \|_{L^r} \le \|Q_N \|_{L^2}^{2/r} \|Q_N \|_{L^\infty}^{1-2/r} \le 6 \kappa \, \big\|\Gamma^N \big\|_{\ell^\infty}^{2/r} \]
which vanishes as $N\to \infty$ due to the choice of $\Gamma^N $.
\end{example}

\section{Proof of  Theorem \ref{thm-existence}}\label{sec-proof-existence}

We devote this section to the proof of Theorem \ref{thm-existence}. In the first subsection we introduce smooth approximations of the stochastic Vlasov equation \eqref{stoch-Vlasov-Ito} and prove some basic estimates, then in the second part we show the weak convergence of approximate solutions by using the compactness method.

\subsection{Smooth approximations and estimates}

Note that the Green function satisfies $\nabla G(-x) = -\nabla G(x)$ for all $x\in \T^3= \big[-\frac12,\frac 12 \big]^3$; moreover, $\nabla G(x)\sim C \frac x{|x|^3}$ as $|x|\to 0$. Due to the singularity of $\nabla G$ near the origin, the existence of weak solutions to \eqref{stoch-Vlasov-Ito} is not straightforward, thus we introduce a regularization of $G$. For $\delta\in (0,1/2)$, let $G^\delta\in C^\infty(\T^3)$ be a symmetric function such that $G^\delta(x)= G(x)$ for $|x|\ge \delta$, and $|G^\delta(x)| \lesssim \delta^{-1}$. Now for fixed $\delta>0$, we consider the approximating stochastic Vlasov equations
  \begin{equation}\label{stoch-Vlasov-approx}
  d f^\delta + v\cdot\nabla_x f^\delta\, dt + (E_{\rho^\delta} + B v\times \bm{e}_3) \cdot\nabla_v f^\delta \, dt + \sum_k \sigma_k \cdot\nabla_v f^\delta \circ d W^k_t =0,
  \end{equation}
where
  \begin{equation}\label{smooth-electric-field}
  E_{\rho^\delta_t}(x)= (\nabla G^\delta\ast \rho^\delta_t)(x)= \int_{\T^3} \nabla G^\delta(x-y) \rho^\delta_t(y)\, dy.
  \end{equation}
We supplement \eqref{stoch-Vlasov-approx} with an initial condition $f^\delta_0\in C_c^\infty (\T^3\times \R^3, \R_+)$.

\begin{proposition}\label{prop-well-posedness-approx}
Assume that $f^\delta_0\in C_c^\infty (\T^3\times \R^3, \R_+)$. Then for any $\delta>0$, the approximate equation \eqref{stoch-Vlasov-approx} admits a unique probabilistically strong solution $f^\delta$, represented as
  $$f^\delta_t(x,v)= f^\delta_0 \big(\Phi^{\delta,-1}_t(x,v) \big), $$
where $\big\{\Phi^{\delta,-1}_t \big\}_{t\ge 0}$ is the inverse flow associated to the stochastic equations
  \begin{equation}\label{stoch-characteristics}
  \left\{ \aligned
  d X^\delta_t &= V^\delta_t\, dt, \\
  d V^\delta_t &= \big[E_{\rho^\delta_t}\big(X^\delta_t \big) + B V^\delta_t\times \bm{e}_3 \big]\, dt + \sum_k \sigma_k\big(X^\delta_t \big)\circ dW^k_t
  \endaligned \right.
  \end{equation}
with initial condition $\big(X^\delta_0, V^\delta_0 \big)= (x,v)\in \T^3\times \R^3$. Moreover, $\|f^\delta_t \|_{L^p_{x,v}} = \|f^\delta_0 \|_{L^p_{x,v}}$ for any $p\ge 1$.
\end{proposition}

\begin{proof}
This result can be proved by following the method in \cite[Theorem 13]{CogFla}. Note that, since $\nabla G^\delta\in C_b^\infty(\T^3,\R^3)$, the electric field in \eqref{smooth-electric-field} is smooth on $\T^3$ for any $\rho^\delta_t \in L^1(\T^3)$. Combined with our assumptions on the noise, we see that the stochastic equations \eqref{stoch-characteristics} generate a stochastic flow of diffeomorphisms, cf. \cite{Kunita}. The preservation of $L^p_{x,v}$-norm is due to the measure-preserving property of the stochastic flow $\big\{\Phi^{\delta}_t \big\}_{t\ge 0}$, a fact following from the stochastic version of Liouville theorem.
\end{proof}

\begin{remark}\label{rem-velocity-eq}
Since $d\big(\sigma_k\big(X^\delta_t \big) \big)= V^\delta_t\cdot\nabla \sigma_k\big(X^\delta_t \big)\, dt$, one has $d\big[\sigma_k\big(X^\delta \big), W^k \big]_t=0$, thus the second equation in \eqref{stoch-characteristics} has the equivalent It\^o form
  $$d V^\delta_t = \big[E_{\rho^\delta_t}\big(X^\delta_t \big) + B V^\delta_t\times \bm{e}_3 \big]\, dt + \sum_k \sigma_k\big(X^\delta_t \big)\, dW^k_t . $$
As a result, the quadratic variation is
  $$\frac{d}{dt}\big[V^\delta, V^\delta \big]_t = \sum_k \big|\sigma_k\big(X^\delta_t \big) \big|^2 = {\rm Tr}(Q(0)) = 6\kappa, $$
where the last step is due to \eqref{covariance}.
\end{remark}

Introduce the potential energy
  $$\aligned
  \mathcal V_\delta(f_0) &= -\int_{\T^3}\! \int_{\T^3} G^\delta(x-y) \rho_0(x) \rho_0(y)\,dx dy\\
  &= -\int_{\T^3\times \R^3}\! \int_{\T^3\times \R^3} G^\delta(x-y) f_0(x,v) f_0(y,w)\,dx dv dy dw\\
  &= -\int_{\T^6}\! \int_{\R^6} G^\delta(x-y) f_0(x,v) f_0(y,w)\,dx dv dy dw.
  \endaligned $$
Note that $-G^\delta$ is uniformly bounded from below; thus for any probability density $f_0: \T^3\times \R^3 \to \R_+$, $\mathcal V_\delta(f_0)\ge C$ for some constant independent of $\delta>0$.

In view of the assumptions on $f_0$ in Theorem \ref{thm-existence}, we make the following assumptions on the initial data $\{f^\delta_0 \}_{\delta>0}$ of equation \eqref{stoch-Vlasov-approx}:
  \begin{equation}\label{initial-condition}
  \sup_{\delta>0} \big[\|f^\delta_0 \|_{L^1_{x,v}\cap L^3_{x,v}} \vee \mathcal K(f^\delta_0) \vee \mathcal V_\delta(f^\delta_0)\big] \le C_0 <\infty.
  \end{equation}

In order to derive a relation between $\mathcal K(f^\delta_t)$ and $\mathcal V_\delta(f^\delta_t)$, we first prove the following estimate which is not uniform in $\delta>0$.

\begin{lemma}\label{lem-kinetic}
For any $\delta>0$, there are constants $C_\delta, C_{\kappa,\delta}>0$ such that for any $t\ge 0$, it holds
  $$\E \mathcal K(f^\delta_t) \le  e^{ C_{\delta,\kappa} t} \big[\mathcal K(f^\delta_0) + \|f^\delta_0\|_{L^1_{x,v}}\big]. $$
\end{lemma}

\begin{proof}
We have, by Proposition \ref{prop-well-posedness-approx},
  \begin{equation}\label{lem-kinetic.0}
  \mathcal K(f^\delta_t) =\int_{\T^3\times \R^3} |v|^2 f^\delta_0 \big(\Phi^{\delta,-1}_t(x,v) \big)\, dx dv = \int_{\T^3\times \R^3} \big|V^\delta_t(x,v) \big|^2 f^\delta_0 (x,v)\, dx dv,
  \end{equation}
where in the last step we have used the measure-preserving property of the stochastic flow $\Phi^{\delta}_t= \big(X^\delta_t, V^\delta_t \big)$, $t\ge 0$. In the sequel, we omit sometimes the initial condition $(x,v)$ in $\Phi^{\delta}_t(x,v)= \big(X^\delta_t(x,v), V^\delta_t(x,v) \big)$.

By It\^o's formula and Remark \ref{rem-velocity-eq},
  \begin{equation}\label{lem-kinetic.1}
  \aligned
  d \big|V^\delta_t \big|^2 &= 2V^\delta_t \cdot d V^\delta_t + d\big[V^\delta, V^\delta \big]_t \\
  &= 2V^\delta_t \cdot E_{\rho^\delta_t} \big(X^\delta_t \big) \, dt + 2 \sum_k V^\delta_t \cdot \sigma_k\big(X^\delta_t \big)\, dW^k_t + 6\kappa\, dt,
  \endaligned
  \end{equation}
where we have used $V^\delta_t \cdot \big(V^\delta_t\times \bm{e}_3\big)=0$. Note that $|\nabla G^\delta(x)| \le C_\delta<+\infty$, thus
  $$\big|E_{\rho^\delta_t}(x) \big| \le \int |\nabla G^\delta(x-y)| \rho^\delta_t(y)\, dy \le C_\delta \|\rho^\delta_t \|_{L^1_x} = C_\delta \|f^\delta_t \|_{L^1_{x,v}}  = C_\delta \|f^\delta_0 \|_{L^1_{x,v}} \le C'_\delta . $$
Let $N_t= N_t(x,v)= 2 \sum_k \int_0^t V^\delta_s \cdot \sigma_k\big(X^\delta_s \big)\, dW^k_s$ be the martingale part, with quadratic variation
  $$[N]_t = 4\sum_k \int_0^t \big(V^\delta_s \cdot \sigma_k\big(X^\delta_s \big)\big)^2\, ds \le 24 \kappa \int_0^t \big|V^\delta_s \big|^2\, ds, $$
where the last step follows from \eqref{covariance}. Introduce a sequence of stopping times $\tau_n =\tau_n(x,v) =\inf \big\{t\ge 0: \big|V^\delta_t(x,v) \big|\ge n \big\}$ which are increasing in $n\ge 1$; then $N_{t\wedge \tau_n}$ are square integrable martingales. Thus, by \eqref{lem-kinetic.1},
  $$\aligned
  \E \big|V^\delta_{t\wedge \tau_n} \big|^2 &\le |v|^2 + 2\, \E \int_0^{t\wedge \tau_n}  \big(C'_{\delta} \big|V^\delta_s \big| + 3\kappa \big)\, ds \\
  &\le |v|^2 + C_{\delta,\kappa} \int_0^t \E \big(\big|V^\delta_{s\wedge \tau_n} \big|^2 + 1\big)\, ds.
  \endaligned $$
Gronwall's inequality implies
  $$\E \big|V^\delta_{t\wedge \tau_n} \big|^2 \le (|v|^2 + 1) e^{ C_{\delta,\kappa} t}.$$
Letting $n\to \infty$ and by Fatou's lemma, we obtain
  $$\E \big|V^\delta_t(x,v) \big|^2 \le (|v|^2 + 1) e^{ C_{\delta,\kappa} t}.$$
The desired result follows by combining it with \eqref{lem-kinetic.0}.
\end{proof}

Now we can derive the relation between $\mathcal K(f^\delta_t)$ and $\mathcal V_\delta(f^\delta_t)$.

\begin{proposition}[Energy evolution]\label{prop-kinetic}
Fix $\delta>0$. For any $t>0$, it holds
  \begin{equation}\label{prop-kinetic.0}
  \mathcal K(f^\delta_t) + \mathcal V_\delta (f^\delta_t)= \mathcal K(f^\delta_0) + \mathcal V_\delta(f^\delta_0) + 6\kappa t \|f^\delta_0\|_{L^1_{x,v}} +M_t,
  \end{equation}
where $M_t$ is a square integrable martingale with quadratic variation
  $$[M]_t \le 24\kappa \|f^\delta_0 \|_{L^1_{x,v}} \int_0^t\mathcal K(f^\delta_s)\, ds. $$
\end{proposition}

\begin{proof}
By \eqref{lem-kinetic.0} and \eqref{lem-kinetic.1},
  \begin{equation}\label{prop-kinetic.1}
  \aligned
  \mathcal K(f^\delta_t) &= \mathcal K(f^\delta_0) + 6\kappa t \|f^\delta_0\|_{L^1_{x,v}} + 2\int_0^t\int_{\T^3\times \R^3} V^\delta_s \cdot E_{\rho^\delta_s} \big(X^\delta_s \big) f^\delta_0 (x,v) \,dx dv ds \\
  &\quad + 2 \sum_k \int_0^t\int_{\T^3\times \R^3} V^\delta_s \cdot \sigma_k\big(X^\delta_s \big) f^\delta_0 (x,v) \, dx dv dW^k_s.
  \endaligned
  \end{equation}
It remains to analyze the last two terms in \eqref{prop-kinetic.1}. Recalling \eqref{smooth-electric-field},
  $$\aligned
  & \int_{\T^3\times \R^3} V^\delta_s \cdot E_{\rho^\delta_s} \big( X^\delta_s \big) f^\delta_0 (x,v) \,dx dv \\
  &= \int_{\T^3\times \R^3} V^\delta_s(x,v) \cdot \bigg[\int_{\T^3\times \R^3} \nabla G^\delta\big(X^\delta_s(x,v) -y \big) f^\delta_s(y,w)\, dydw \bigg] f^\delta_0 (x,v) \,dx dv \\
  &= \int_{\T^3\times \R^3} V^\delta_s(x,v) \cdot \bigg[\int_{\T^3\times \R^3} \nabla G^\delta\big(X^\delta_s(x,v) - X^\delta_s(y,w) \big) f^\delta_0(y,w)\, dydw \bigg] f^\delta_0 (x,v) \,dx dv ,
  \endaligned $$
where the last step is due to Proposition \ref{prop-well-posedness-approx}. Exchanging variables $(x,v) \leftrightarrow (y,w)$, we have
  $$\aligned
  & 2 \int_{\T^3\times \R^3} V^\delta_s \cdot E_{\rho^\delta_s} \big( X^\delta_s \big) f^\delta_0 (x,v) \,dx dv \\
  &= \int_{\T^3\times \R^3} V^\delta_s(x,v) \cdot \bigg[\int_{\T^3\times \R^3} \nabla G^\delta\big(X^\delta_s(x,v) - X^\delta_s(y,w) \big) f^\delta_0(y,w)\, dydw \bigg] f^\delta_0 (x,v) \,dx dv \\
  &\quad + \int_{\T^3\times \R^3} V^\delta_s(y,w) \cdot \bigg[\int_{\T^3\times \R^3} \nabla G^\delta\big(X^\delta_s(y,w) - X^\delta_s(x,v) \big) f^\delta_0(x,y)\, dx dv \bigg] f^\delta_0 (y,w) \, dydw \\
  &= \int_{\T^6}\! \int_{\R^6} \big(V^\delta_s(x,v)- V^\delta_s(y,w)\big)\cdot  \nabla G^\delta\big(X^\delta_s(x,v) - X^\delta_s(y,w) \big) f^\delta_0(y,w) f^\delta_0 (x,v) \,dx dv dydw,
  \endaligned $$
where we have used the property $\nabla G^\delta(-z)= -\nabla G^\delta(z)$.

Next, letting $M_t$ be the martingale part in \eqref{prop-kinetic.1}, then
  $$\aligned
  d[M]_t &=4 \sum_k \bigg(\int_{\T^3\times \R^3} V^\delta_s \cdot \sigma_k\big(X^\delta_s \big) f^\delta_0 (x,v) \, dx dv \bigg)^2 d t\\
  &\le 4\|f^\delta_0 \|_{L^1_{x,v}} \sum_k \int_{\T^3\times \R^3} \big| V^\delta_s \cdot \sigma_k\big(X^\delta_s \big) \big|^2 f^\delta_0 (x,v) \, dx dv d t\\
  &\le 24\kappa \|f^\delta_0 \|_{L^1_{x,v}} \int_{\T^3\times \R^3} \big| V^\delta_s \big|^2 f^\delta_0 (x,v) \, dx dv dt,
  \endaligned $$
where the last step is due to \eqref{covariance}. By \eqref{lem-kinetic.0} and the assertion of Lemma \ref{lem-kinetic}, we see that $\{M_t \}_{t\ge 0}$ is a square integrable martingale whose quadratic variation has the desired estimate.

Finally, we consider the potential energy:
  $$\aligned
  \mathcal V_\delta (f^\delta_t) &= -\int_{\T^6}\! \int_{\R^6} G^\delta(x-y) f^\delta_t(x,v) f^\delta_t(y,w)\,dx dv dy dw \\
  &= -\int_{\T^6}\! \int_{\R^6} G^\delta\big(X^\delta_t(x,v) - X^\delta_t(y,w)\big)f^\delta_0(y,w) f^\delta_0 (x,v) \,dx dv dydw.
  \endaligned $$
By the first equation in \eqref{stoch-characteristics}, we have
  $$\aligned
  & G^\delta\big(X^\delta_t(x,v) - X^\delta_t(y,w)\big) \\
  &= G^\delta(x-y) + \int_0^t \nabla G^\delta \big(X^\delta_s(x,v) - X^\delta_s(y,w)\big)\cdot \big(V^\delta_s(x,v) - V^\delta_s(y,w)\big)\, ds,
  \endaligned $$
as a result,
  $$\aligned
  \mathcal V_\delta (f^\delta_t) &= \mathcal V_\delta(f^\delta_0) - \int_0^t\!\! \int_{\T^6}\! \int_{\R^6}\! \nabla G^\delta \big(X^\delta_s(x,v) - X^\delta_s(y,w)\big)\cdot \big(V^\delta_s(x,v) - V^\delta_s(y,w)\big)\, dx dv dydw ds.
  \endaligned $$
Combining this identity with the computations below \eqref{prop-kinetic.1}, we arrive at the desired result.
\end{proof}

The next estimate will play an important role in the proof of tightness of laws of $\{f^\delta\}_{\delta>0}$.

\begin{corollary}\label{cor-kinetic-energy}
Under \eqref{initial-condition}, for any $q\ge 1$, we have
  $$\sup_{\delta>0} \sup_{t\in [0,T]} \E\, \mathcal K(f^\delta_t)^{2q} \le C_{q,T} <+\infty. $$
\end{corollary}

\begin{proof}
We know that $\mathcal V_\delta (f^\delta_t) \ge -C_1>-\infty$, thus by \eqref{initial-condition} and Proposition \ref{prop-kinetic}, we have
  $$\mathcal K(f^\delta_t) \le C_1 + C_0(2+ 3\kappa t) + M_t.$$
As a result, for any $q\ge 1$,
  $$\mathcal K(f^\delta_t)^{2q} \le 2^{2q-1}\big[ C_{\kappa}^{2q} (1+t)^{2q} + M_t^{2q} \big]. $$
Introducing a stopping time $\tau_n:= \inf\big\{t>0: \mathcal K(f^\delta_t)\ge n \big\}$ which is almost surely finite by Lemma \ref{lem-kinetic}; then,
  \begin{equation}\label{cor-kinetic-energy.1}
  \mathcal K(f^\delta_{t\wedge \tau_n})^{2q} \le 2^{2q-1}\big[ C_{\kappa}^{2q} (1+t)^{2q} + M_{t\wedge \tau_n}^{2q} \big]. \end{equation}
By Burkholder-Davis-Gundy inequality and Cauchy's inequality,
  $$\aligned
  \E |M_{t\wedge \tau_n}|^{2q} &\le C_q\, \E\bigg[ \int_0^{t\wedge \tau_n} \sum_k \bigg( \int_{\T^3\times \R^3} V^\delta_s \cdot \sigma_k\big(X^\delta_s \big) f^\delta_0 (x,v) \, dx dv \bigg)^2\, ds \bigg]^q \\
  &\le C_q \|f^\delta_0 \|_{L^1_{x,v}}^q \E \bigg[ \int_0^{t\wedge \tau_n} \sum_k \int_{\T^3\times \R^3} \big[ V^\delta_s \cdot \sigma_k\big(X^\delta_s \big) \big]^2 f^\delta_0 (x,v) \, dx dv ds \bigg]^q \\
  &\le C'_q \kappa^q \|f^\delta_0 \|_{L^1_{x,v}}^q \E\bigg[ \int_0^{t\wedge \tau_n} \int_{\T^3\times \R^3} \big| V^\delta_s \big|^2 f^\delta_0 (x,v) \, dx dv ds \bigg]^q,
  \endaligned $$
where the last step follows from Remark \ref{rem-velocity-eq}. The definition of the kinetic energy implies
  $$\aligned
  \E |M_{t\wedge \tau_n}|^{2q} &\le C'_q \kappa^q \|f^\delta_0 \|_{L^1_{x,v}}^q \E\bigg[ \int_0^{t\wedge \tau_n} \mathcal K(f^\delta_s)\, ds \bigg]^q \\
  &\le C'_q \kappa^q \|f^\delta_0 \|_{L^1_{x,v}}^q \E\bigg[ \int_0^t \mathcal K(f^\delta_{s\wedge \tau_n})\, ds \bigg]^q \\
  &\le C_{q,\kappa} t^{q-1} \E \int_0^t \mathcal K(f^\delta_{s\wedge \tau_n})^q\, ds \\
  &\le \frac{C_{q,\kappa} t^{q-1}}2 \bigg( t + \int_0^t \E\, \mathcal K(f^\delta_{s\wedge \tau_n})^{2q}\, ds \bigg).
  \endaligned $$
Combining this estimate with \eqref{cor-kinetic-energy.1}, and using Gronwall's inequality, we obtain
  $$\E\, \mathcal K(f^\delta_{t\wedge \tau_n})^{2q} \le C_{1,q,T} e^{C_{2,q,T}} \quad \mbox{for all } t\in [0,T]. $$
By Fatou's lemma, letting $n\to \infty$ leads to the desire result.
\end{proof}

\subsection{Tightness argument and weak convergence of approximate solutions} \label{subsec-tightness}

Next we turn to proving the tightness of laws of the family $\{f^\delta \}_{\delta>0}$ in $C\big([0,T], H^{-\epsilon}_{x,v, loc} \big)$. By Proposition \ref{prop-well-posedness-approx} and the assumption \eqref{initial-condition}, we already know that $\P$-a.s., $f^\delta\in L^\infty\big(0,T; L^2_{x,v} \big)$; thus by Simon's compactness embedding result \cite{Simon}, it suffices to show that $\{f^\delta \}_{\delta>0}$ is bounded in probability in $C^\alpha \big([0,T], H^{-4}_{x,v} \big)$ for some $\alpha\in (0,1/2)$.

First, we recall the definition of Sobolev space $H^a_{x,v}= H^a(\T^3\times \R^3)$ for $a\in \R$. Let $f\in L^1( \T^3\times \R^3, \R)$ with Fourier transform
  $$\hat f(l,\xi) =\mathcal F(f)(l,\xi) = \int_{\T^3\times \R^3} f(x,v) e^{-il\cdot x - i \xi\cdot v} \, dx dv, \quad l\in \Z^3, \xi\in \R^3.$$
We say that $f\in H^a_{x,v}$ if
  $$\|f \|_{H^a_{x,v}}^2 = \sum_l \int (1+|l|^2 +|\xi|^2)^a |\hat f(l,\xi)|^2\, d\xi <+\infty. $$

\begin{lemma}\label{lem-Holder}
For all $0\le s< t\le T$, it holds that
  $$\E \|f^\delta_t- f^\delta_s \|_{H^{-4}_{x,v}}^4 \le C|t-s|^2. $$
\end{lemma}

\begin{proof}
By the equation for $f^\delta_t$, we have
  $$\aligned
  f^\delta_t- f^\delta_s &= -\int_s^t v\cdot \nabla_x f^\delta_r\, dr - \int_s^t (E_{\rho^\delta_r} + Bv\times \bm{e}_3) \cdot \nabla_v f^\delta_r\, dr \\
  &\quad + \kappa \int_s^t \Delta_v f^\delta_r\, dr - \sum_k \int_s^t \sigma_k \cdot \nabla_v f^\delta_r\, d W^k_r.
  \endaligned $$
We denote the four terms on the right-hand side by $I_i$, $i=1,2,3,4$. First,
  $$\|I_1 \|_{H^{-4}_{x,v}}^4 \le |t-s|^3 \int_s^t \|v\cdot \nabla_x f^\delta_r \|_{H^{-4}_{x,v}}^4\, dr. $$
By definition,
  $$\mathcal F(v\cdot \nabla_x f^\delta_r)(l,\xi)= i \int_{\T^3\times \R^3} (v\cdot l) f^\delta_r(x,v) e^{-il\cdot x - i \xi\cdot v} \, dx dv; $$
Cauchy's inequality implies
  $$\big|\mathcal F(v\cdot \nabla_x f^\delta_r)(l,\xi) \big|^2 \le \bigg(\int_{\T^3\times \R^3} |v\cdot l| f^\delta_r(x,v) \, dx dv \bigg)^2 \le |l|^2 \|f^\delta_r\|_{L^1_{x,v}} \mathcal K(f^\delta_r). $$
As a result,
  $$\|v\cdot \nabla_x f^\delta_r \|_{H^{-4}_{x,v}}^2 \le \sum_l \int (1+|l|^2 +|\xi|^2)^{-4} |l|^2 \|f^\delta_r\|_{L^1_{x,v}} \mathcal K(f^\delta_r) \, d\xi \le C \|f^\delta_0 \|_{L^1_{x,v}} \mathcal K(f^\delta_r),$$
where $C$ is a dimensional constant. Therefore, by \eqref{initial-condition}, we have
  \begin{equation}\label{lem-Holder.1}
  \E \|I_1 \|_{H^{-4}_{x,v}}^4 \lesssim |t-s|^3 \int_s^t \E\, \mathcal K(f^\delta_r)^2 \, dr \lesssim_T |t-s|^4,
  \end{equation}
where the last step is due to Corollary \ref{cor-kinetic-energy}.

Concerning the term $I_2$, we first consider the part $I_{2,1}$ involving $E_{\rho^\delta_r}$:
  $$\|I_{2,1} \|_{H^{-4}_{x,v}}^4 \le |t-s|^3 \int_s^t \big\|E_{\rho^\delta_r}\cdot \nabla_v f^\delta_r \big\|_{H^{-4}_{x,v}}^4\, dr. $$
We have
  $$\mathcal F\big(E_{\rho^\delta_r}\cdot \nabla_v f^\delta_r \big)(l,\xi)= i \int_{\T^3\times \R^3} \big(E_{\rho^\delta_r}(x) \cdot \xi \big) f^\delta_r(x,v) e^{-il\cdot x - i \xi\cdot v} \, dx dv, $$
thus
  $$\aligned
  \big| \mathcal F\big(E_{\rho^\delta_r}\cdot \nabla_v f^\delta_r \big)(l,\xi) \big|
  &\le |\xi| \int_{\T^3\times \R^3} \big|E_{\rho^\delta_r}(x) \big| f^\delta_r(x,v) \, dx dv \\
  &= |\xi| \int_{\T^3\times \R^3} \big| \nabla G^\delta \ast\rho^\delta_r (x) \big| f^\delta_r(x,v) \, dxdv  \\
  &\le |\xi| \int_{(\T \times \R)^6} \frac{C}{|x-y|^2} f^\delta_r(y,w) f^\delta_r(x,v)\, dx dv dy dw,
  \endaligned $$
where $C$ is independent of $\delta>0$. Applying Lemma \ref{app-lem-product}(ii) with $p=3$, we obtain
  \begin{equation}\label{eq-L-3}
  \big| \mathcal F\big(E_{\rho^\delta_r}\cdot \nabla_v f^\delta_r \big)(l,\xi) \big| \lesssim |\xi|\, \|f^\delta_r\|_{L^3_{x,v}} \mathcal K(f^\delta_r) = |\xi|\, \|f^\delta_0\|_{L^3_{x,v}} \mathcal K(f^\delta_r) .
  \end{equation}
By \eqref{initial-condition}, we have
  $$\big\|E_{\rho^\delta_r}\cdot \nabla_v f^\delta_r \big\|_{H^{-4}_{x,v}}^2 \lesssim \sum_l \int (1+|l|^2 +|\xi|^2)^{-4}  |\xi|^2 \mathcal K(f^\delta_r)^2 \, d\xi \lesssim \mathcal K(f^\delta_r)^2. $$
As a consequence,
  $$\E \|I_{2,1} \|_{H^{-4}_{x,v}}^4 \lesssim |t-s|^3 \E \int_s^t \mathcal K(f^\delta_r)^4\, dr \lesssim_T |t-s|^4, $$
where the last step is again due to Corollary \ref{cor-kinetic-energy}. The estimate of the part involving $v\times \bm{e}_3$ in $I_2$ is similar, indeed, we have
  $$\big| \mathcal F\big((v\times \bm{e}_3)\cdot \nabla_v f^\delta_r \big)(l,\xi) \big| \le |\xi|\, \|f^\delta_0\|_{L^1_{x,v}}^{1/2} \mathcal K(f^\delta_r)^{1/2}$$
and thus $\big\|(v\times \bm{e}_3)\cdot \nabla_v f^\delta_r \big\|_{H^{-4}_{x,v}}^2 \lesssim \mathcal K(f^\delta_r) $, which immediately gives us the desired estimate.

Next, the treatment of $I_3$ is very easy; in fact, one can easily show that $\big\|\Delta_v f^\delta_r \big\|_{H^{-4}_{x,v}}^2 \le \|f^\delta_0\|_{L^1_{x,v}}$ which, combined with \eqref{initial-condition}, implies $\P$-a.s.,
  $$\|I_3 \|_{H^{-4}_{x,v}}^4 \lesssim |t-s|^4. $$

Finally, we deal with the stochastic term $I_4$:
  $$\E\|I_4 \|_{H^{-4}_{x,v}}^4 \lesssim \E \bigg(\int_s^t \sum_k \big\|\sigma_k \cdot \nabla_v f^\delta_r \big\|_{H^{-4}_{x,v}}^2\, dr \bigg)^2. $$
One has $\big| \mathcal F\big(E_{\rho^\delta_r}\cdot \nabla_v f^\delta_r \big)(l,\xi) \big| \lesssim |\xi| \int |\sigma_k(x)| f^\delta_r(x,v)\, dx dv$, and thus by Cauchy's inequality,
  $$\big| \mathcal F\big(E_{\rho^\delta_r}\cdot \nabla_v f^\delta_r \big)(l,\xi) \big|^2 \lesssim |\xi|^2 \|f^\delta_r\|_{L^1_{x,v}} \int |\sigma_k(x)|^2 f^\delta_r(x,v)\, dx dv. $$
Consequently,
  $$\aligned
  \sum_k \big\|\sigma_k \cdot \nabla_v f^\delta_r \big\|_{H^{-4}_{x,v}}^2
  &\lesssim \sum_k \sum_l \int(1+|l|^2 +|\xi|^2)^{-4} |\xi|^2 \int |\sigma_k(x)|^2 f^\delta_r(x,v)\, dx dv d\xi \\
  &\le 6\kappa \sum_l \int(1+|l|^2 +|\xi|^2)^{-3} \int f^\delta_r(x,v)\, dx dv d\xi \\
  &\lesssim \kappa \|f^\delta_0\|_{L^1_{x,v}} \le C \kappa
  \endaligned $$
by assumption \eqref{initial-condition}. Therefore, $\E\|I_4 \|_{H^{-4}_{x,v}}^4 \lesssim \kappa^2 |t-s|^2$. Summarizing the above estimates, we complete the proof.
\end{proof}

Thanks to Lemma \ref{lem-Holder} and the discussions at the beginning of Section \ref{subsec-tightness}, we deduce from Simon's compactness theorem \cite{Simon} that the family of laws of $\{f^\delta\}_{\delta>0}$ is tight in
  $$\mathcal X:= C\big([0,T], H^{-\epsilon }_{x,v, loc} \big)$$
for any $\epsilon>0$. By Prohorov's theorem, we can find a subsequence $\delta_n\to 0$ as $n\to \infty$, such that the laws $\eta_n$ of $f^{\delta_n}$ converges weakly, in the topology of $\mathcal X$, to some probability measure $\eta\in \mathcal P(\mathcal X)$. By Skorohod's representation theorem, there exist a new probability space $\big(\tilde\Omega, \tilde{\mathcal{F}}, \tilde\P \big)$, a sequence of stochastic processes $\big\{\tilde f^n \big\}_{n\ge 1}$ and a limit process $\tilde f$, such that
\begin{itemize}
\item[(i)] for any $n\ge 1$, $\tilde f^n$ has the same law $\eta_n$ as $f^{\delta_n}$;
\item[(ii)] $\tilde\P$-a.s., $\tilde f^n$ converges in the topology of $\mathcal X$ to $\tilde f$.
\end{itemize}
By item (i), we know that for all $n\ge 1$ and $p\in [1,2]$,
  $$\tilde\P \mbox{-a.s.}, \quad \big\|\tilde f^n_t \big\|_{L^p_{x,v}} = \big\|f^{\delta_n}_t \big\|_{L^p_{x,v}}= \big\|f^{\delta_n}_0 \big\|_{L^p_{x,v}} \quad \mbox{for all } t\in [0,T], $$
and for all $q\ge 1$,
  $$\tilde\E \mathcal K\big(\tilde f^n_t \big)= \E \mathcal K\big(f^{\delta_n}_t \big) \le C_{q,T}\quad \mbox{for all } t\in [0,T] . $$
Thanks to \eqref{initial-condition},  passing to a further subsequence if necessary, we can deduce related bounds for the limit process $\{\tilde f_t\}_{t\in [0,T]}$; namely, for all $p\in [1,2]$,
  $$\tilde\P \mbox{-a.s.}, \quad \big\|\tilde f_t \big\|_{L^p_{x,v}} \le C_0 <\infty \quad \mbox{for all } t\in [0,T] $$
and for all $q\ge 1$,
  $$\tilde\E \mathcal K\big(\tilde f_t \big) \le C_{q,T}\quad \mbox{for all } t\in [0,T] . $$
Therefore, the properties (a) and (b) in Theorem \ref{thm-existence} are satisfied.

With slightly more effort, we can also find a sequence of Brownian motions $\big\{ \tilde B^{n,k} \big\}_{n,k\ge 1}$ and limit Brownian motions $\{\tilde B^{k}\}_{k\ge 1}$ on $\tilde\Omega$, such that
  \begin{equation}\label{eq-law}
  \big(\tilde f^n, ( \tilde B^{n,k} )_{k\ge 1}\big)\ \stackrel{\mathcal L}{\sim}\ \big(f^{\delta_n}, ( W^k )_{k\ge 1}\big) \quad \mbox{for all } n\ge 1,
  \end{equation}
and $\tilde\P$-a.s., for any $k\ge 1$, $\tilde B^{n,k}$ converges in $C([0,T],\R)$ to $\tilde B^{k}$. We omit the details here, cf. \cite[Section 2]{FGL21a}.

Now we start proving that $\big(\tilde f, \{\tilde B^{k}\}_{k\ge 1}\big)$ is a weak solution to \eqref{stoch-Vlasov-Ito}. Let $\<\cdot, \cdot\>$ be the inner product in $L^2(\T^3\times \R^3)$. First, for any $\phi\in C_c^\infty (\T^3\times \R^3)$, by \eqref{eq-law} and the It\^o formulation of \eqref{stoch-Vlasov-approx}, we know that $\tilde\P$-a.s. for all $t\ge 0$, it holds
  $$\aligned
  \big\<\tilde f^n_t, \phi\big\> &= \big\<f^{\delta_n}_0, \phi\big\> + \int_0^t \big\<\tilde f^n_s, v\cdot\nabla_x \phi \big\>\, ds + \int_0^t \big\<\tilde f^n_s, \big(E_{\tilde \rho^n_s} + B(v\times \bm{e}_3) \big) \cdot\nabla_v \phi \big\> \, ds\\
  &\quad + \sum_k \int_0^t \big\<\tilde f^n_s, \sigma_k\cdot\nabla_v \phi \big\>\, d \tilde B^{n,k}_s + \kappa \int_0^t \big\<\tilde f^n_s, \Delta_v \phi \big\>\, ds,
  \endaligned $$
where $\tilde\rho^n_s(x)= \int \tilde f^n_s(x,v)\, dv$. Thanks to the above discussions, as $n\to\infty$, all the terms converge to the corresponding ones involving the limit process $\tilde f$, except the nonlinear part
  $$J_n:= \int_0^t \big\<\tilde f^n_s, E_{\tilde \rho^n_s} \cdot\nabla_v \phi \big\> \, ds.$$
Thus we focus on showing the convergence of $J_n$, which can be further expressed as
  $$J_n:= \int_0^t\! \int_{(\T\times\R)^6} \nabla G^{\delta_n}(x-y)\cdot \nabla_v\phi(x,v)\tilde f^n_s(x,v)\tilde f^n_s(y,w) \, dx dv dy dw ds,$$
where $(\T\times\R)^6= (\T^3\times \R^3) \times (\T^3\times \R^3)$. Let
  $$J:= \int_0^t\! \int_{(\T\times\R)^6} \nabla G(x-y)\cdot \nabla_v\phi(x,v)\tilde f_s(x,v)\tilde f_s(y,w) \, dx dv dy dw ds,$$
then
  $$\aligned
  J_n-J &= \int_0^t\! \int_{(\T\times\R)^6} \big(\nabla G^{\delta_n}(x-y)- \nabla G(x-y) \big)\cdot \nabla_v\phi(x,v)\tilde f^n_s(x,v)\tilde f^n_s(y,w) \, dx dv dy dw ds\\
  &\quad +\int_0^t\! \int_{(\T\times\R)^6} \nabla G(x-y)\cdot \nabla_v\phi(x,v) \big(\tilde f^n_s(x,v) - \tilde f_s(x,v) \big) \tilde f^n_s(y,w) \, dx dv dy dw ds\\
  &\quad +\int_0^t\! \int_{(\T\times\R)^6} \nabla G(x-y)\cdot \nabla_v\phi(x,v) \tilde f_s(x,v)\big(\tilde f^n_s(y,w) - \tilde f_s(y,w) \big) \, dx dv dy dw ds\\
  &=: J_n^1 + J_n^2 + J_n^3.
  \endaligned $$

First, letting $\tilde\rho^{n,\phi}_s (x)= \int_{\R^3} \nabla_v\phi(x,v)\, \tilde f^n_s(x,v)\, dv$ where the integral is in fact on a compact subset in $\R^3$, then
  $$J_n^1= \int_0^t\! \int_{\R^3} \tilde\rho^{n,\phi}_s (x) (\nabla G^{\delta_n} - \nabla G)\ast \tilde\rho^n_s (x)\, dx ds; $$
by Cauchy's inequality and Young's inequality,
  $$\aligned
  |J_n^1|&\le \int_0^t \big\| \tilde\rho^{n,\phi}_s \big\|_{L^2} \big\| (\nabla G^{\delta_n} - \nabla G)\ast \tilde\rho^n_s \big\|_{L^2} \, ds \\
  &\le \int_0^t \big\| \tilde\rho^{n,\phi}_s \big\|_{L^2} \big\|\nabla G^{\delta_n} - \nabla G \big\|_{L^q} \big\| \tilde\rho^n_s \big\|_{L^{r(2)}} \, ds,
  \endaligned $$
where the parameters satisfy $1+ \frac12= \frac1q + \frac1{r(2)}$ and, by Lemma \ref{app-lem-density}, $r(2)= \frac75$; therefore, $q= \frac{14}{11} < \frac32$. Using Lemmas \ref{app-lem-compact} and \ref{app-lem-density}, we have
  $$\big\| \tilde\rho^{n,\phi}_s \big\|_{L^2}\le C_\phi \big\| \tilde f^{n}_s \big\|_{L^2}, \quad \big\| \tilde\rho^n_s \big\|_{L^{r(2)}} \le C' \big\| \tilde f^n_s \big\|_{L^{2}}^{4/7} \mathcal K\big(\tilde f^n_s \big)^{3/7}, $$
hence,
  $$\aligned
  |J_n^1| &\le C'_\phi \big\|\nabla G^{\delta_n} - \nabla G \big\|_{L^{14/11}} \int_0^t \big\| \tilde f^n_s \big\|_{L^{2}}^{11/7} \mathcal K\big(\tilde f^n_s \big)^{3/7} \, ds \\
  &= C'_\phi \big\|\nabla G^{\delta_n} - \nabla G \big\|_{L^{14/11}} \big\| f^{\delta_n}_0 \big\|_{L^{2}}^{11/7} \int_0^t \mathcal K\big(\tilde f^n_s \big)^{3/7} \, ds.
  \endaligned $$
By the properties of $G$, we know that $\big\|\nabla G^{\delta_n} - \nabla G \big\|_{L^{14/11}}\to 0$ as $n\to \infty$; combining with \eqref{initial-condition} and Corollary \ref{cor-kinetic-energy}, we conclude that
  $$\lim_{n\to \infty} \tilde \E|J_n^1|=0. $$

Next, $J_n^2$ can be rewritten as
  $$\aligned
  J_n^2 &= \int_0^t\! \int_{\T^3\times\R^3} \big(\tilde f^n_s(x,v) - \tilde f_s(x,v) \big)\nabla_v\phi(x,v) \cdot \bigg[\int_{\T^3} \nabla G(x-y) \tilde \rho^n_s(y) dy\bigg] \, dx dv ds \\
  &= \int_0^t\! \int_{\T^3\times\R^3} \big(\tilde f^n_s(x,v) - \tilde f_s(x,v) \big)\nabla_v\phi(x,v) \cdot E_{\tilde \rho^n_s}(x) \, dx dv ds.
  \endaligned $$
Note that $\nabla_v\phi$ has compact support, thus we can regard $\tilde f^n_s - \tilde f_s$ as a compactly supported distribution, and
  \begin{equation}\label{eq-J-n-2}
  \aligned
  |J_n^2| &\le \int_0^t\!  \big\|\tilde f^n_s - \tilde f_s \big\|_{H^{-\epsilon}_{x,v,loc}} \big\|\nabla_v\phi \cdot E_{\tilde \rho^n_s} \big\|_{H^\epsilon_{x,v}} ds.
  \endaligned
  \end{equation}
We have
  $$\aligned
  \big\|\nabla_v\phi \cdot E_{\tilde \rho^n_s} \big\|_{H^\epsilon_{x,v}}^2
  &= \sum_l \int (1+|l|^2 +|\xi|^2)^\epsilon \big|\mathcal F\big(\nabla_v\phi \cdot E_{\tilde \rho^n_s} \big)(l,\xi) \big|^2\, d\xi \\
  &= \sum_l \int (1+|l|^2 +|\xi|^2)^\epsilon \big|\mathcal F_x\big(E_{\tilde \rho^n_s}\cdot \mathcal F_v (\nabla_v\phi(x,\cdot))(\xi) \big)(l) \big|^2\, d\xi,
  \endaligned $$
where $\mathcal F_v (\nabla_v\phi(x,\cdot))(\xi)$ means the Fourier transform of $\nabla_v\phi(x,\cdot)$ with respect to the velocity variable $v$, similarly for $\mathcal F_x$. Therefore,
  $$\aligned
  \big\|\nabla_v\phi \cdot E_{\tilde \rho^n_s} \big\|_{H^\epsilon_{x,v}}^2
  &\le \int (1+|\xi|^2)^\epsilon \sum_l (1+|l|^2)^\epsilon \big|\mathcal F_x\big(E_{\tilde \rho^n_s}\cdot \mathcal F_v (\nabla_v\phi(x,\cdot))(\xi) \big)(l) \big|^2\, d\xi \\
  &= \int (1+|\xi|^2)^\epsilon \big\| E_{\tilde \rho^n_s}\cdot \mathcal F_v (\nabla_v\phi(x,\cdot))(\xi) \big\|_{H^\epsilon_x}^2\, d\xi \\
  &\lesssim_\epsilon \int (1+|\xi|^2)^\epsilon \big\| E_{\tilde \rho^n_s}\big\|_{H^{2\epsilon}_x}^2 \big\| \mathcal F_v (\nabla_v\phi(x,\cdot))(\xi) \big\|_{H^{3/2- \epsilon}_x}^2\, d\xi,
  \endaligned $$
where in the last step we have used the product rule of Sobolev functions. As a result,
  $$\aligned
  \big\|\nabla_v\phi \cdot E_{\tilde \rho^n_s} \big\|_{H^\epsilon_{x,v}}^2
  &\le \big\| E_{\tilde \rho^n_s}\big\|_{H^{2\epsilon}_x}^2 \int (1+|\xi|^2)^\epsilon \sum_l (1+|l|^2)^{3/2-\epsilon} \big| \mathcal F_x\big(\mathcal F_v (\nabla_v\phi(x,\cdot))(\xi) \big)(l) \big|^2\, d\xi \\
  &\le \big\| E_{\tilde \rho^n_s}\big\|_{H^{2\epsilon}_x}^2 \sum_l \int (1+ |l|^2+ |\xi|^2)^{3/2} \big| \mathcal F_{x,v} (\nabla_v\phi)(l,\xi) \big|^2\, d\xi \\
  &= \big\| E_{\tilde \rho^n_s}\big\|_{H^{2\epsilon}_x}^2 \| \nabla_v\phi \|_{H^{3/2}_{x,v}}^2.
  \endaligned $$
Substituting this estimate into \eqref{eq-J-n-2} leads to
  $$|J_n^2| \lesssim \int_0^t\!  \big\|\tilde f^n_s - \tilde f_s \big\|_{H^{-\epsilon}_{x,v,loc}} \big\| E_{\tilde \rho^n_s}\big\|_{H^{2\epsilon}_x} \| \nabla_v\phi \|_{H^{3/2}_{x,v}}\, ds. $$

Since $\int E_{\tilde \rho^n_s}(x) \, dx=0$, we have
  $$\big\| E_{\tilde \rho^n_s} \big\|_{H^{2\epsilon}_{x}}^2\lesssim \sum_{l\ne 0} |l|^{4\epsilon} \big|\mathcal F(\nabla G\ast \tilde \rho^n_s)(l) \big|^2 \lesssim \sum_{l\ne 0} |l|^{4\epsilon-2} \big|\mathcal F(\tilde \rho^n_s)(l) \big|^2 \lesssim \|\tilde \rho^n_s \|_{H^{2\epsilon-1}_{x}}^2. $$
By Sobolev embedding, $L^{7/5}_x \subset H^{-9/14}_{x} \subset H^{2\epsilon-1}_{x}$ for $\epsilon \in (0,5/28]$, thus
  $$\big\| E_{\tilde \rho^n_s} \big\|_{H^{2\epsilon}_{x}} \lesssim \|\tilde \rho^n_s \|_{L^{7/5}_{x}} \lesssim \big\|\tilde f^n_s \big\|_{L^2_{x}}^{4/7} \mathcal K\big(\tilde f^n_s\big)^{3/7}, $$
where the last step follows from Lemma \ref{app-lem-density}. Summarizing these discussions, we obtain
  $$\aligned
  |J_n^2| &\le \|\nabla_v\phi \|_{H^{3/2- \epsilon}_{x,v}} \big\|f^{\delta_n}_0 \big\|_{L^2_{x}}^{4/7} \int_0^t \big\|\tilde f^n_s - \tilde f_s \big\|_{H^{-\epsilon}_{x,v,loc}} \mathcal K\big(\tilde f^n_s\big)^{3/7}\, ds.
  \endaligned $$
By assumption \eqref{initial-condition}, Corollary \ref{cor-kinetic-energy} and item (ii) above, we deduce that
  $$\lim_{n\to \infty} \tilde \E|J_n^2|=0. $$

Finally, the convergence of $J_n^3$ follows from the weak convergence of $\tilde f^n$ to $\tilde f$, up to a further subsequence, thus we complete the proof of weak existence of solutions to \eqref{stoch-Vlasov-Ito}.

\section{Proof of Theorem \ref{thm-limit}}\label{sec-proof-limit}

By Theorem \ref{thm-existence}, the solutions $\{f^N \}_N$ to \eqref{stoch-Vlasov-Ito-N} verify the following bounds, uniformly in $N$:
\begin{itemize}
\item[\rm(a')] for any $N\ge 1$ and $p\in [1, 3]$, $\P$-a.s. $\|f^N_t\|_{L^p_{x,v}} \le \|f_0 \|_{L^p_{x,v}}$ for all $t\in [0,T]$;
\item[\rm(b')] for any $N\ge 1$ and $q\ge 1$, $\E \mathcal K(f^N_t)^{2q} \le C_{q,T}$  for all $t\in [0,T]$.
\end{itemize}
Note that the solutions $\{f^N \}_N$ might be defined on different probability spaces $\big(\Omega^N, \mathcal F^N, \P^N \big)$, but we do not distinguish the notations $\P^N $ or $\E^N$ for simplicity of notation.

Thanks to the above bounds, using the equation \eqref{stoch-Vlasov-Ito-N} and following the computations in the proof of Lemma \ref{lem-Holder}, we can show that the laws of the family $\{f^N\}_N$ are tight on
  $$\mathcal X:= C\big([0,T], H^{-\epsilon }_{x,v, loc} \big).$$
Therefore, we can repeat the arguments at the end of the last section to show that, along a subsequence, $\{f^N\}_N$ converges weakly, in the topology of $\mathcal X$, to some limit $\bar f$. To show that $\bar f$ solves \eqref{limit-Vlasov} in the weak sense, we use the weak formulation of  \eqref{stoch-Vlasov-Ito-N}: for any test function $\phi\in C_c^\infty(\T^3\times \R^3)$,
  $$\aligned
  \big\<f^N_t, \phi\big\> &= \big\<f_0, \phi\big\> + \int_0^t \big\<f^N_s, v\cdot\nabla_x \phi \big\>\, ds + \int_0^t \big\<f^N_s, \big(E_{\rho^N_s} + B(v\times \bm{e}_3) \big) \cdot\nabla_v \phi \big\> \, ds\\
  &\quad + \sum_k \int_0^t \big\<f^N_s, \sigma^N_k\cdot\nabla_v \phi \big\>\, d W^k_s + \kappa \int_0^t \big\<f^N_s, \Delta_v \phi \big\>\, ds,
  \endaligned $$
where $\rho^N_t(x)= \int f^N_t(x,v)\, dv$. Since all the terms, except the martingale part, can be treated similarly as in Section 2, here we only show that the martingale part
  $$\bar M_t:= \sum_k\int_0^t \big\<f^N_s, \sigma^N_k \cdot\nabla_v \phi\big\>\, d W^k_s$$
vanishes in a suitable sense. The argument is by now classical: by the It\^o isometry,
  $$\aligned
  \E \bar M_t^2 &= \sum_k \E \int_0^t \big\<f^N_s, \sigma^N_k \cdot\nabla_v \phi\big\>^2\, d s \\
  &= \sum_k \E \int_0^t\! \int_{(\T\times \R)^6} \big(f^N_s \nabla_v \phi\big)(x,v) \sigma^N_k(x)\otimes \sigma^N_k(y) \big(f^N_s \nabla_v \phi\big)(y,w)\,dx dv dy dw d s \\
  &= \E \int_0^t\! \int_{(\T\times \R)^6} \big(f^N_s \nabla_v \phi\big)(x,v)^\ast Q_N(x-y) \big(f^N_s \nabla_v \phi\big)(y,w)\,dx dv dy dw d s,
  \endaligned $$
where $\ast$ means transposition of vectors and the last step is due to \eqref{covariance-N}. Letting
  $$A^N_s(x):= \int_{\R^3} \big(f^N_s \nabla_v \phi\big)(x,v)\, dv, \quad (s,x)\in [0,T]\times \T^3, $$
then by H\"older's inequality and Young's inequality,
  $$\aligned
  \E \bar M_t^2& = \E \int_0^t\! \int_{\T^3} A^N_s(x)^\ast\, \big(Q_N\ast A^N_s \big)(x)\, dx ds\\
  &\le  \E \int_0^t \big\| A^N_s \big\|_{L^{7/5}_x} \big\|Q_N\ast A^N_s \big\|_{L^{7/2}_x} \, ds \\
  &\le \|Q_N \|_{L^{7/4}_x}\, \E \int_0^t \big\| A^N_s \big\|_{L^{7/5}_x}^2 \, ds.
  \endaligned $$
Note that
  $$\big\| A^N_s \big\|_{L^{7/5}_x} \le \| \nabla_v \phi\|_{L^\infty_{x,v}} \big\| \rho^N_s \big\|_{L^{7/5}_x} \lesssim_\phi \big\| f^N_s \big\|_{L^{2}_x}^{4/7}\, \mathcal K\big(f^N_s \big)^{3/7}, $$
where we have used Lemma \ref{app-lem-density} in the last step. As a result,
  $$\aligned
  \E \bar M_t^2&\lesssim_\phi \|Q_N \|_{L^{7/4}_x}\, \E \int_0^t \big\| f^N_s \big\|_{L^{2}_x}^{8/7}\, \mathcal K\big(f^N_s \big)^{6/7} \, ds \\
  &\le \|Q_N \|_{L^{7/4}_x} \| f_0 \|_{L^{2}_x}^{8/7} \int_0^t \E\, \mathcal K\big(f^N_s \big)^{6/7} \, ds \\
  &\lesssim_T \|Q_N \|_{L^{7/4}_x} \to 0
  \endaligned  $$
as $N\to \infty$, thanks to \eqref{condition-covariance}.

\begin{remark}
As mentioned before, we cannot prove that the limit equation \eqref{limit-Vlasov} admits a unique weak solution in the class of functions satisfying (a') and (b'), therefore, we do not know whether the whole sequence $\{f^N\}_N$ converges to the same limit.
\end{remark}

\section{A particular noise in Vlasov equations} \label{sec-particular-noise}

In this section we describe the origin and motivation for introducing a noise
$\partial_{t}W (x,t)  $, with certain properties, in the Vlasov
equation%
\begin{equation} \label{Vlasov stochastic}
\partial_{t}f+ \div_x(v f) + \div_v((E + v\times B)f) + \div_{v}(\partial_{t}W(
x,t)  \circ f)  =0.
\end{equation}
A noise in that particular transport form may come either from fluctuations of
the magnetic field, or fluctuations of the electric field (or both). We
concentrate our attention on the electric field only. In the case of the
magnetic field, it could be reasonable to assume that the noise is divergence
free (similarly to many recent investigations of transport noise in fluid
mechanic models, see e.g. \cite{Gal, FGL21a}). But being related to the electric field, it is a gradient in
the space variable, hence we cannot assume it to be divergence free.

Before we enter into some detail, let us explain a few elements of the idea.
The noise comes from the electric field%
\[
E\left(  x,t\right)  =\int_{\mathbb{T}^{3}}\nabla G\left(  x-y\right)
\rho\left(  y,t\right)  dy,
\]
where $G$ is the Green kernel on $\mathbb{T}^{3}$. We may think that
$\rho\left(  y,t\right)  $ is composed of a slowly varying part $\overline
{\rho}\left(  y,t\right)  $ plus a rapidly fluctuating one $\rho^{\prime
}\left(  y,t\right)  $, so that the same holds for the electric field:
\begin{align*}
E\left(  x,t\right)   &  =\int_{\mathbb{T}^{3}}\nabla G\left(  x-y\right)
\overline{\rho}\left(  y,t\right)  dy
+\int_{\mathbb{T}^{3}}\nabla G\left(  x-y\right)  \rho^{\prime}\left(
y,t\right)  dy.
\end{align*}
Thus it would be natural to assume that $\rho^{\prime}\left(  y,t\right)  $ is
a noise. However, a noise is typically a zero-mean (the mean of the signal
being incorporated into $\overline{\rho}\left(  y,t\right)  $). In the
particular case of the space density $\rho\left(  y,t\right)  $, it is not so
clear that its natural fluctuations are zero mean. A superficial reason is
that $\rho\left(  y,t\right)  \geq0$, thus decomposing it into $\overline
{\rho}\left(  y,t\right)  \geq0$ and $\rho^{\prime}\left(  y,t\right)  $ would
impose a complicated constraint on $\rho^{\prime}\left(  y,t\right)  $. More deeply, the model of $\rho^{\prime}\left(  y,t\right)  $ should correspond to the information we have on the coherent structures of electrons
and ions, information that is accumulating in the physical literature, see for
instance \cite{Hutch1, Hutch2}, and may lead us to consider
non-zero-mean perturbations $\rho^{\prime}\left(  y,t\right)  $. Therefore, we develop a scheme which
does not require $\rho^{\prime}\left(  y,t\right)  $ to be zero mean. The key
remark is that, assuming $\rho^{\prime}\left(  y,t\right)  $
space-homogeneous, the term
\[
E^{\prime}\left(  x,t\right)  =\int_{\mathbb{T}^{3}}\nabla G\left(
x-y\right)  \rho^{\prime}\left(  y,t\right)  dy
\]
is zero mean and thus may have the properties of a noise.

We divide the section into two subsections. The first one 5.1 illustrates how
a single perturbation of non-zero-average generates a zero-mean electric field
$E^{\prime}(x,t)$, and we compute its covariance operator, giving also an
example of a Gaussian random field with the same covariance operator. The
second one 5.2 shows how such objects may be used as building blocks of a
white-noise-in-time Gaussian process, that will be our noise $W(x,t)$ in
equation \eqref{Vlasov stochastic}.

\subsection{Preliminaries on random fields\label{subsect random fields}}

Assume we are on the torus $\mathbb{T}^{3}$. The Coulomb Green kernel, up to
constants, is $G\left(  x\right)  =\frac{1}{\left\vert x\right\vert }+h\left(
x\right)  $ for a smooth function $h$, and the associated electric field
kernel is $\nabla G\left(  x\right)  = - \frac{x}{\left\vert x\right\vert ^{3}%
}+\nabla h\left(  x\right)  $. One has the following properties, beside the
fact that the singularity is integrable:%
\begin{align*}
&  \nabla G\left(  -x\right)   =-\nabla G\left(  x\right) \quad \mbox{and}\quad
 \nabla G\text{ periodic.}%
\end{align*}
The Fourier transforms are%
\[
\widehat{G}\left(  k\right)  =\frac{1}{\left\vert k\right\vert ^{2}}%
,\qquad\widehat{\nabla G}\left(  k\right)  =\frac{ik}{\left\vert k\right\vert
^{2}},\qquad k\in\mathbb{Z}_{0}^{3},%
\]
where $\mathbb{Z}_{0}^{3}$ are the non-zero elements of $\mathbb{Z}^{3}$.

Consider a perturbation $\rho^{\prime}\left(  x\right)  $ (at some time $t$)
of the space-density $\rho\left(  x,t\right)  $ of the form%
\begin{align*}
\rho^{\prime}\left(  x\right)   &  =r\theta_{\ell}\left(  x-x_{0}\right),
\end{align*}
where $r>0$, $x_0\in \T^3$, and $\theta_{\ell}\left(  x\right) =\ell^{-3}\theta\left(
\ell^{-1}x\right)$ for some square integrable periodic function $\theta$. It is a paradigmatic
(idealized) choice for modeling a blob, a concentration of density, possibly
surrounded by an annulus of negative values (a Mexican hat form of $\theta$);
a blob localized around $x_{0}$, with size (space-scale) $\ell$ and intensity
$r$. We may randomize all parameters, $r,\ell,x_{0}$, and consider a random
blob of the form%
\[
\rho^{\prime}\left(  x\right)  =R\,\theta_{L}\left(  x-X\right),
\]
where $R,L$ and $X$ are independent random variables, $X$ is uniform in $\mathbb{T}^{3}%
$, $L$ is positive with density $f$, $R$ is real valued with square average
$\mathbb{E}\left[  R^{2}\right]  =\sigma^{2}$.

Assume, in addition to $\theta$ being square integrable and $f$ a density,
that
\[
\sum_{k\in\mathbb{Z}_{0}^{3}}\frac1{\left\vert k\right\vert ^{2}} \int_{0}^{\infty} \Big\vert \widehat{\theta
_{\ell}^{T}}\left(  k\right)  \Big\vert^{2}f\left(  \ell\right)  d\ell
<\infty.
\]
This condition will be trivially satisfied by our special choice of $f$
(compact support disjoined from zero) in the scaling limit, which is our main
concern in the paper.

Define the random vector field%
\[
E\left(  x\right)  =R\nabla\left(  G\ast\theta_{L}\right)  \left(  x-X\right)
.
\]
As it is clear from the computations below, under the assumption just made we
have%
\[
\mathbb{E}\big[  \left\langle E,\phi\right\rangle ^{2}\big]  <\infty
\]
for every $\phi\in L^{2}\left(  \mathbb{T}^{3},\mathbb{R}^{3}\right)  $, hence
it is well defined at least as a generalized random field.

\begin{lemma}
For every $\phi\in L^{2}\left(  \mathbb{T}^{3},\mathbb{R}^{3}\right)  $,%
\[
\mathbb{E}\left[  \left\langle E,\phi\right\rangle \right]  =0.
\]

\end{lemma}

\begin{proof}%
By definition,
\begin{align*}
\mathbb{E}\left[  \left\langle E,\phi\right\rangle \right]   &  =\int%
_{\mathbb{T}^{3}}\int_{\mathbb{T}^{3}}\phi\left(  x\right)  \cdot\nabla
G\left(  x-y\right)  \mathbb{E}\left[  R\left(  \int_{\mathbb{T}^{3}}%
\theta_{L}\left(  y-z\right)  dz\right)  \right]  dydx\\
&  =\left\Vert \theta\right\Vert _{L^{1}}\mathbb{E}\left[  R\right]
\int_{\mathbb{T}^{3}}\int_{\mathbb{T}^{3}}\phi\left(  x\right)  \cdot\nabla
G\left(  x-y\right)  dydx\\
&= \left\Vert \theta\right\Vert _{L^{1}}\mathbb{E}\left[  R\right]\int_{\mathbb{T}^{3}}\phi\left(  x\right)  \cdot
\int_{\mathbb{T}^{3}}\nabla G\left(  -y\right)  dydx=0.
\end{align*}

\end{proof}

\begin{lemma}
For every $\phi,\psi\in L^{2}\left(  \mathbb{T}^{3},\mathbb{R}^{3}\right)  $,
\begin{equation}\label{covariance-electric-field}
\mathbb{E}\left[  \left\langle E,\phi\right\rangle \left\langle E,\psi
\right\rangle \right]  =\sigma^{2}\sum_{k\in\mathbb{Z}_{0}^{3}}\frac{\chi
^{2}\left(  k\right)  }{\left\vert k\right\vert ^{2}}\left\langle
\frac{k\otimes k}{\left\vert k\right\vert ^{2}}\widehat{\phi}\left(  k\right)
,\overline{\widehat{\psi}\left(  k\right)  }\right\rangle
\end{equation}
with
\[
\chi\left(  k\right)  =\sqrt{\int_{0}^{\infty}\left\vert \widehat{\theta
_{\ell}^{T}}\left(  k\right)  \right\vert ^{2}f\left(  \ell\right)  d\ell},%
\]
where $\theta_{\ell}^{T}\left(  x\right)  =\theta_{\ell}\left(  -x\right)  $.
Namely, the covariance function is
\[
Q\left(  x-y\right)  =\sigma^{2}\sum_{k\in\mathbb{Z}_{0}^{3}}\frac{\chi
^{2}\left(  k\right)  }{\left\vert k\right\vert ^{2}}\frac{k\otimes
k}{\left\vert k\right\vert ^{2}}e^{2\pi ik\cdot\left(  x-y\right)  }.
\]

\end{lemma}

\begin{proof}%
We have
\begin{align*}
\left\langle E,\phi\right\rangle  &  =R\int_{\mathbb{T}^{3}}\int%
_{\mathbb{T}^{3}}\phi\left(  x\right)  \cdot\nabla G\left(  x-y\right)
\theta_{L}\left(  y-X\right)  dydx\\
&  =R\big(  \theta_{L}^{T}\ast\left(  \nabla G\right)  ^{T}\ast\phi\big)
\left(  X\right),
\end{align*}
where $\left(  \nabla G\right)  ^{T}\left(  x\right)  =\nabla G\left(
-x\right)  $. Hence%
\begin{align*}
\mathbb{E}\left[  \left\langle E,\phi\right\rangle \left\langle E,\psi
\right\rangle \right]
&  =\sigma^{2}\mathbb{E}\int_{\mathbb{T}^{3}} \big(
\theta_{L}^{T}\ast\left(  \nabla G\right)  ^{T}\ast\phi\big)  \left(
x\right)  \big(  \theta_{L}^{T}\ast\left(  \nabla G\right)  ^{T}\ast
\psi\big)  \left(  x\right)  dx\\
&  =\sigma^{2}\sum_{k\in\mathbb{Z}_{0}^{3}}\mathbb{E}\left[  \left\vert
\widehat{\theta_{L}^{T}}\left(  k\right)  \right\vert^{2} \right]  \frac
{1}{\left\vert k\right\vert ^{2}}\left\langle \frac{k\otimes k}{\left\vert
k\right\vert ^{2}}\widehat{\phi}\left(  k\right)  ,\overline{\widehat{\psi
}\left(  k\right)  }\right\rangle .
\end{align*}
Using the density function for $L$, we obtain the desired result.
\end{proof}

Consider the space $L_{0}^{2}\left(  \mathbb{T}^{3},\mathbb{R}^{3}\right)  $
(mean zero periodic $L^{2}$-fields), defined on the complex field. A
complete orthonormal system is $\left(  e^{2\pi ik\cdot x}\right)
_{k\in\mathbb{Z}_{0}^{3}}$.

\begin{corollary}
The centered random vector field $E$ (random element of $L_{0}^{2}\left(
\mathbb{T}^{3},\mathbb{R}^{3}\right)  $) has the same covariance function as the
Gaussian centered random vector field%
\[
F\left(  x\right)  =\sigma\sum_{k\in\mathbb{Z}_{0}^{3}}\frac{\chi\left(
k\right)  }{\left\vert k\right\vert }\frac{k}{\left\vert k\right\vert }e^{2\pi
ik\cdot x}Z_{k},%
\]
where $\{Z_{k}\}_k$ are i.i.d. standard normal random variables.
\end{corollary}

\begin{proof}%
One has
\begin{align*}
\mathbb{E}\big( \langle F,\phi \rangle \overline{\left\langle
F,\psi\right\rangle }\big)
&  =\sigma^{2}\sum_{k\in\mathbb{Z}_{0}^{3}%
}\frac{\chi^{2}\left(  k\right)  }{\left\vert k\right\vert ^{2}} \frac
{1}{\left\vert k\right\vert ^{2}} \langle ke^{2\pi ik\cdot x}%
,\phi \rangle \overline{\left\langle ke^{2\pi ik\cdot x},\psi
\right\rangle }\\
&  =\sigma^{2}\sum_{k\in\mathbb{Z}_{0}^{3}}\frac{\chi^{2}\left(  k\right)
}{\left\vert k\right\vert ^{2}}\frac{1}{\left\vert k\right\vert ^{2}}%
k\cdot\widehat{\phi}\left(  k\right)  k\cdot\overline{\widehat{\psi}\left(
k\right)  }\\
&  =\sigma^{2}\sum_{k\in\mathbb{Z}_{0}^{3}}\frac{\chi^{2}\left(  k\right)
}{\left\vert k\right\vert ^{2}}\left\langle \frac{k\otimes k}{\left\vert
k\right\vert ^{2}}\widehat{\phi}\left(  k\right)  ,\overline{\widehat{\psi
}\left(  k\right)  }\right\rangle,
\end{align*}
which is the same as the right-hand side of \eqref{covariance-electric-field}.
\end{proof}

Let us go back to $L_{0}^{2}\left(  \mathbb{T}^{3},\mathbb{R}^{3}\right)  $
considered on the field of real numbers. Split $\mathbb{Z}_{0}^{3}$ in two
parts, called $\mathbb{Z}_{+}^{3},\mathbb{Z}_{-}^{3}$, as in Example \ref{exa-noises}. Define%
\begin{align*}
F_{\operatorname{real}}\left(  x\right)   &  =\sqrt{2\sigma^{2}}\sum
_{k\in\mathbb{Z}_{+}^{3}}\frac{\chi\left(  k\right)  }{\left\vert k\right\vert
}\frac{k}{\left\vert k\right\vert }\cos\left(  2\pi k\cdot x\right)  Z_{k}\\
&  +\sqrt{2\sigma^{2}}\sum_{k\in\mathbb{Z}_{-}^{3}}\frac{\chi\left(  k\right)
}{\left\vert k\right\vert }\frac{k}{\left\vert k\right\vert }\sin\left(  2\pi
k\cdot x\right)  Z_{k}%
\end{align*}
where, again, $\{Z_{k}\}_k$ are i.i.d. standard normal random variables. We have%
\begin{align*}
 \mathbb{E}\left[  F\left(  x\right) F\left(  y\right)  \right]
&  =2\sigma^{2}\sum_{k\in\mathbb{Z}_{+}^{3}}\frac{\chi^{2}\left(  k\right)
}{\left\vert k\right\vert ^{2}}\frac{k\otimes k}{\left\vert k\right\vert ^{2}%
}\cos\left(  2\pi k\cdot x\right)  \cos\left(  2\pi k\cdot y\right)  \\
&  +2\sigma^{2}\sum_{k\in\mathbb{Z}_{-}^{3}}\frac{\chi^{2}\left(  k\right)
}{\left\vert k\right\vert ^{2}}\frac{k\otimes k}{\left\vert k\right\vert ^{2}%
}\sin\left(  2\pi k\cdot x\right)  \sin\left(  2\pi k\cdot y\right) \\
&=2\sigma^{2}\sum_{k\in\mathbb{Z}_{+}^{3}}\frac{\chi^{2}\left(  k\right)
}{\left\vert k\right\vert ^{2}}\frac{k\otimes k}{\left\vert k\right\vert ^{2}%
}\cos\left(  2\pi k\cdot\left(  x-y\right)  \right),
\end{align*}%
hence%
\begin{align*}
\mathbb{E}\left[  F\left(  x\right) F\left(  y\right)  \right]   &
=2\sigma^{2}\sum_{k\in\mathbb{Z}_{+}^{3}}\frac{\chi^{2}\left(  k\right)
}{\left\vert k\right\vert ^{2}}\frac{k\otimes k}{\left\vert k\right\vert ^{2}%
}\left(  \frac{e^{2\pi ik\cdot\left(  x-y\right)  }+e^{-2\pi ik\cdot\left(
x-y\right)  }}{2}\right)  \\
&  =\sigma^{2}\sum_{k\in\mathbb{Z}_{0}^{3}}\frac{\chi^{2}\left(  k\right)
}{\left\vert k\right\vert ^{2}}\frac{k\otimes k}{\left\vert k\right\vert ^{2}%
}e^{2\pi ik\cdot\left(  x-y\right)  }%
\end{align*}
which gives the same result as the complex case.

\subsection{The noise and Donsker type
argument\label{subsect noise and Donsker}}

Based on the previous subsection and the Donsker type argument described
below, we shall choose in equation (\ref{Vlasov stochastic}) the following
space-dependent Brownian motion%
\begin{align*}
W\left(  x,t\right)   &  =\sqrt{2\tau\sigma^{2}}\sum_{k\in\mathbb{Z}_{+}^{3}%
}\frac{\chi\left(  k\right)  }{\left\vert k\right\vert }\frac{k}{\left\vert
k\right\vert }\cos\left(  2\pi k\cdot x\right)  W_{t}^{k}\\
&  +\sqrt{2\tau\sigma^{2}}\sum_{k\in\mathbb{Z}_{-}^{3}}\frac{\chi\left(
k\right)  }{\left\vert k\right\vert }\frac{k}{\left\vert k\right\vert }%
\sin\left(  2\pi k\cdot x\right)  W_{t}^{k},%
\end{align*}
where $\tau$ is described
below and $\left\{  W_{t}^{k}\right\} _{k\in\mathbb{Z}_{0}^{3}}$ is a family of
independent standard Brownian motions. Its covariance function is%
\[
Q\left(  x-y\right)  =2\tau\sigma^{2}\sum_{k\in\mathbb{Z}_{+}^{3}}\frac
{\chi^{2}\left(  k\right)  }{\left\vert k\right\vert ^{2}}\frac{k\otimes
k}{\left\vert k\right\vert ^{2}}\cos\left(  2\pi k\cdot\left(  x-y\right)
\right)  .
\]

Let us motivate this choice of noise. We start from a space density
$\rho\left(  x,t\right)  $, and we assume that there is a slowly varying
background density $\overline{\rho}\left(  x,t\right)  $, a sequence of increasing times $\left(
t_{n}\right)  _{n=0,1,\ldots}$ and perturbations $\left(  \rho_{n}\left(
x\right)  \right)  _{n=0,1,\ldots}$ such that $\rho\left(  x,t\right)  $ in the
time interval $\left[  t_{n},t_{n+1}\right)  $ is equal to $\overline{\rho
}\left(  x,t\right)  +\rho_{n}\left(  x\right)  $; namely%
\[
\rho\left(  x,t\right)  =\overline{\rho}\left(  x,t\right)  +\sum_{n}\rho
_{n}\left(  x\right)  1_{\left[  t_{n},t_{n+1}\right)  }\left(  t\right)  .
\]

Suppose for simplicity that there is $\tau>0$ such that $t_{n+1}-t_{n}=\tau$
(but the conclusion is similar if, for instance, we assume that the intertimes
$t_{n+1}-t_{n}$ are independent and exponentially distributed random variables
with mean $\tau$). We assume that $\rho_{n}\left(  x\right)  $ are independent
random fields of the form%
\[
\rho_{n}\left(  x\right)  =R_{n}\theta_{L_{n}}\left(  x-X_{n}\right),
\]
where all the random variables $\left(  R_{n},L_{n},X_{n}\right)  _{n=0,1,\ldots}$ are
independent, $X_{n}$ are uniformly distributed in $\mathbb{T}^{3}$, $R_{n}$
and $L_{n}$ are positive with probability density functions $g\left(
r\right)  $ and $f\left(  \ell\right)  $ respectively, $\theta$ is an
integrable function (not necessarily positive), $\theta_{\ell}\left(
x\right)  =\ell^{-3}\theta\left(  \ell^{-1}x\right)  $. The intuition is that
the perturbations have the form $R\theta_{\ell}\left(  x-x_{0}\right)  $,
where the size $R$, the space-scale $\ell$ and the center $x_{0}$ are randomly distributed.

The corresponding electric potential $\phi\left(  x,t\right)  $ is given by%
\begin{align*}
\phi\left(  x,t\right)  &=\overline{\phi}\left(  x,t\right)  +\phi^{\prime
}\left(  x,t\right) \quad \mbox{with} \\
\overline{\phi}\left(  \cdot,t\right)   &  =\left(  -\Delta\right)  ^{-1}%
\rho\left(  \cdot,t\right), \\
\phi^{\prime}\left(  \cdot,t\right)   &  =\sum_{n}\left(  -\Delta\right)
^{-1}\rho_{n}1_{\left[  t_{n},t_{n+1}\right)  }\left(  t\right)  .
\end{align*}
Namely, $\phi\left(  x,t\right)  $ is the sum of a background slowly varying
electric potential $\left(  -\Delta\right)  ^{-1}\rho\left(  \cdot,t\right)  $
plus a fast varying perturbation. For the link with experiments and physical
arguments, it may be possible that the measured fluctuations are for the
electric potential, instead of the space density.

The corresponding electric field $E\left(  x,t\right)  $ is composed of a
background slowly varying electric field
\[
\overline{E}\left(  \cdot,t\right)  =\nabla\left(  -\Delta\right)  ^{-1}%
\rho\left(  \cdot,t\right)
\]
plus a fast varying perturbation
\[
E^{\prime}\left(  \cdot,t\right)  =\sum_{n}\nabla\left(  -\Delta\right)
^{-1}\rho_{n}1_{\left[  t_{n},t_{n+1}\right)  }\left(  t\right)  .
\]
In our preliminary modeling phase (still only heuristic), we replace
$E^{\prime}\left(  \cdot,t\right)  $ by the time derivative of a
space-dependent Brownian motion:%
\[
E^{\prime}\left(  \cdot,t\right)  \longrightarrow \sqrt{\tau
}\partial_{t}W\left(  \cdot,t\right)
\]
multiplied by a factor $\sqrt{\tau}$ which gives us the correct unit of
measure and keeps a trace of the correlation time of the original random
fluctuations. The heuristic procedure for this replacement is that, given
$T>0$ of the form $T=N\tau$, dropping $x$ for notational simplicity,
\[\aligned
\frac{1}{\sqrt{\tau}}\int_{0}^{T}E^{\prime}\left(  t\right)  dt
&=\frac{1}%
{\sqrt{\tau}}\sum_{t_{n+1}\leq T}\nabla\left(  -\Delta\right)  ^{-1}\rho
_{n}\left(  t_{n+1}-t_{n}\right)  +\frac{1}{\sqrt{\tau}}\nabla\left(
-\Delta\right)  ^{-1}\rho_{n+1}\left(  T-t_{n+1}\right) \\
&=\sqrt{\frac TN}\sum_{n\leq N-1}\nabla\left(  -\Delta\right)  ^{-1}\rho_{n}%
\endaligned \]
which converges (by the Central Limit Theorem) to a Gaussian random variable for given
$T$; and (by Donsker theorem) to a Brownian motion $W_{T}$ when $T$ varies,
with covariance operator $\mathbb{Q}$ in $L^{2}\left(  \mathbb{T}%
^{3},\mathbb{R}^{3}\right)  $ given by
\[
\left\langle \mathbb{Q}v,w\right\rangle _{L^{2}\left(  \mathbb{T}%
^{3},\mathbb{R}^{3}\right)  }
=\mathbb{E}\left[  \big\langle \nabla\left(
-\Delta\right)  ^{-1}\rho^{\prime},v \big\rangle _{L^{2}\left(
\mathbb{T}^{3},\mathbb{R}^{3}\right)  } \big\langle \nabla\left(
-\Delta\right)  ^{-1}\rho^{\prime}, w\big\rangle _{L^{2}\left(
\mathbb{T}^{3},\mathbb{R}^{3}\right)  }\right],
\]
where $\rho^{\prime}$ is distributed as $\rho_{n}$. Essential for the result
is the property%
\[
\mathbb{E}\big[  \nabla\left(  -\Delta\right)  ^{-1}\rho^{\prime}\big]  =0,
\]
otherwise we would get, in the scaling limit, an infinite contribution from the mean.

\subsection{Choice for the scaling limit}

We construct now an example of a sequence of Brownian motions which will be
used in the diffusive scaling limit. We prescribe a sequence of random variables $R_{N}$
with $\mathbb{E}\left[  R_{N}^{2}\right]  =\sigma_{N}^{2}$, and a sequence of
functions $f_{N}$ (densities of a sequence $L_{N}$) of the form%
\[
f_{N}\left(  \ell\right)  =\ell_{N}^{-1}1_{\left[  \ell_{N},2\ell_{N}\right]} \left(  \ell\right),
\]
with $\ell_N\to 0$ as $N\to \infty$. The intuition behind this choice is that we have blob-like fluctuations at
space-scale $\ell_{N}$, uniformly and densely distributed on $\mathbb{T}^{3}$,
with small intensity $\sigma_{N}^{2}$. The
noise has the form
\begin{align*}
W_{N}\left(  x,t\right)   &  =\sqrt{2\tau\sigma_{N}^{2}}\sum_{k\in
\mathbb{Z}_{+}^{3}}\frac{\chi_{N}\left(  k\right)  }{\left\vert k\right\vert
}\frac{k}{\left\vert k\right\vert }\cos\left(  2\pi k\cdot x\right)  W_{t}%
^{k}\\
&  +\sqrt{2\tau\sigma_{N}^{2}}\sum_{k\in\mathbb{Z}_{-}^{3}}\frac{\chi
_{N}\left(  k\right)  }{\left\vert k\right\vert }\frac{k}{\left\vert
k\right\vert }\sin\left(  2\pi k\cdot x\right)  W_{t}^{k}, %
\end{align*}
with
\[
\chi_{N}\left(  k\right)  =\sqrt{\ell_{N}^{-1}\int_{\ell_{N}}^{2\ell_{N}%
}\left\vert \widehat{\theta_{\ell}^{T}}\left(  k\right)  \right\vert ^{2}%
d\ell}%
\]
and covariance function%
\[
Q_{N}\left(  x-y\right)  =2\tau\sigma_{N}^{2}\sum_{k\in\mathbb{Z}_{+}^{3}%
}\frac{\chi_{N}^{2}\left(  k\right)  }{\left\vert k\right\vert ^{2}}%
\frac{k\otimes k}{\left\vert k\right\vert ^{2}}\cos\left(  2\pi k\cdot\left(
x-y\right)  \right)  .
\]
In view of the next computation, where $\chi_{N}\left(  k\right)  $ becomes
close to a constant, notice that for finite $N$ this expression is well
defined, just because $\theta$ is assumed to be square integrable.

In the sequel, we assume that $\theta$ is radially symmetric and has compact support on $\T^3$, and thus we can also regard $\theta$ as a function on $\R^3$ which vanishes outside $\big(-\frac12, \frac12\big)^3$.

\begin{lemma}\label{lem-conv-Fourier-coeff}
We have
\[
\lim_{N\rightarrow\infty}\chi_{N}\left(  k\right)  =\left\Vert \theta
\right\Vert _{L^{1}} %
\]
for every $k\in\mathbb{Z}_{0}^{3}$.
\end{lemma}

\begin{proof}
By definition,
  \begin{equation}\label{eq-Fourier-coefficients}
  \aligned
  \widehat{\theta_{\ell}^{T}}\left(  k\right)
  &= \int_{\T^3} \theta_{\ell}^{T}(x) e^{-2\pi ik\cdot x}\, dx
  = \int_{\T^3} \ell^{-3} \theta(-\ell^{-1} x) e^{-2\pi ik\cdot x}\, dx \\
  &= \int_{\T^3_{\ell^{-1}}} \theta(y) e^{-2\pi i \ell k\cdot y}\, dy
  = \int_{\T^3} \theta(y) e^{-2\pi i \ell k\cdot y}\, dy
  = \widehat\theta(-\ell k),
  \endaligned
  \end{equation}
where $\T^3_{\ell^{-1}}= \big[-\frac1{2\ell}, \frac1{2\ell}\big]^3$ and we have used the fact that $\theta$ has compact support in $\big(-\frac12, \frac12\big)^3$.
The desired limit follows immediately from the fact that $\widehat\theta(-\ell k)\to \widehat\theta(0)= \|\theta \|_{L^1}$ as $\ell\to 0$.
\end{proof}

\begin{lemma}\label{lem-converg-covariance}
Assume
\[
\lim_{N\rightarrow\infty}2\sigma_{N}^{2}\sum_{k\in\mathbb{Z}_{+}^{3}}%
\frac{\chi_{N}^{2}\left(  k\right)  }{\left\vert k\right\vert ^{2}}=k_{T}^{2}>0,
\]
then ($I_3$ being the $3\times 3$ unit matrix)
\[
Q_{N}\left(  0\right)  \rightarrow \frac13 \tau k_{T}^{2} I_3.
\]
\end{lemma}

\begin{remark}
The series $\sum_{k\in\mathbb{Z}_{+}^{3}}\frac{\chi_{N}^{2}\left(  k\right)
}{\left\vert k\right\vert ^{2}}$ is well defined (finite numbers)\ and
diverges to infinity as $N\to \infty$. The lemma states that we have to take $\sigma_{N}^{2}$
going to zero as the reciprocal of this sequence, in order to have a finite
non zero limit of the diffusion term.
\end{remark}

\begin{proof}[Proof of Lemma \ref{lem-converg-covariance}]
Recall that we assume $\theta\in L^1(\R^3)$ to be radially symmetric, hence so is the function $\xi\mapsto \big|\widehat\theta(\xi)\big|$; by \eqref{eq-Fourier-coefficients} and the definition of $\chi_{N} \left(  k\right)$, we see that the latter is also a radial function in $k$. As a result,
  $$ Q_{N}\left(0\right)  =2\tau\sigma_{N}^{2}\sum_{k\in\mathbb{Z}_{+}^{3}%
  }\frac{\chi_{N}^{2}\left(  k\right)  }{\left\vert k\right\vert ^{2}}%
  \frac{k\otimes k}{\left\vert k\right\vert ^{2}}
  = \tau\sigma_{N}^{2}\sum_{k\in\mathbb{Z}_0^{3}%
  }\frac{\chi_{N}^{2}\left(  k\right)  }{\left\vert k\right\vert ^{2}}%
  \frac{k\otimes k}{\left\vert k\right\vert ^{2}}
  = \frac13 \tau\sigma_{N}^{2}\sum_{k\in\mathbb{Z}_0^{3}%
  }\frac{\chi_{N}^{2}\left(  k\right)  }{\left\vert k\right\vert ^{2}} I_3 .$$
The last equality is classical; we present the proof here for completeness. For any $i,j\in \{1,2,3\}$ with $i\neq j$, say $i=1, j=2$,
  $$Q_{N}^{1,2}\left(0\right)= \tau\sigma_{N}^{2}\sum_{k\in\mathbb{Z}_0^{3}%
  }\frac{\chi_{N}^{2}\left(  k\right)  }{\left\vert k\right\vert ^{2}}%
  \frac{k_1 k_2}{\left\vert k\right\vert ^{2}}; $$
computing the sum over the four points $(k_1, k_2, k_3)$, $(k_1, -k_2, k_3)$, $(-k_1, k_2, k_3)$ and $(-k_1, -k_2, k_3)$, we see that $Q_{N}^{1,2}\left(0\right)=0$; the same is true for other values of $i\neq j$. Next, by symmetry, one has
  $$Q_{N}^{1,1}\left(0\right)=Q_{N}^{2,2}\left(0\right)=Q_{N}^{3,3}\left(0\right)= \frac13 \tau\sigma_{N}^{2}\sum_{k\in\mathbb{Z}_0^{3}%
  }\frac{\chi_{N}^{2}\left(  k\right)  }{\left\vert k\right\vert ^{2}} \frac{k_1^2+k_2^2+k_3^2}{|k|^2}
  = \frac13 \tau\sigma_{N}^{2}\sum_{k\in\mathbb{Z}_0^{3}%
  }\frac{\chi_{N}^{2}\left(  k\right)  }{\left\vert k\right\vert ^{2}}.$$
This completes the proof.
\end{proof}

\begin{lemma}
Under the assumption in Lemma \ref{lem-converg-covariance}, the covariance function $Q_N$ satisfies the conditions in \eqref{condition-covariance}.
\end{lemma}

\begin{proof}
The first property is satisfied by taking $\kappa= \frac16 \tau k_{T}^{2}$. To show the second property, let
 $$ \Gamma^N_k= \sigma_N \frac{\chi_{N} \left(  k\right)  }{\left\vert k\right\vert}, \quad k\in \Z^3_0;$$
we only need to prove that $\|\Gamma^N \|_{\ell^\infty}\to 0$. Indeed, by \eqref{eq-Fourier-coefficients}, one has for any $k$,
  $$|\chi_{N}^2 ( k )|= \ell_{N}^{-1}\int_{\ell_{N}}^{2\ell_{N}} \big\vert \widehat\theta(-\ell k) \big\vert^{2} d\ell \le \|\theta \|_{L^1}^2, $$
thus $\|\Gamma^N \|_{\ell^\infty} \le \sigma_N \|\theta \|_{L^1}$. It remains to show that $\sigma_N\to 0$ as $N\to \infty$. By definition,
  $$\sigma_N \lesssim \Bigg(\sum_{k\in\mathbb{Z}_{+}^{3}} \frac{\chi_{N}^{2}\left(  k\right)  }{\left\vert k\right\vert ^{2}} \Bigg)^{-1/2}
  \le \Bigg(\sum_{k\in\mathbb{Z}_{+}^{3}, |k|\le M} \frac{\chi_{N}^{2}\left(  k\right)  }{\left\vert k\right\vert ^{2}} \Bigg)^{-1/2} $$
for any $M\ge 1$. By Lemma \ref{lem-conv-Fourier-coeff}, letting $N\to \infty$ in the above inequality yields
  $$\limsup_{N\to\infty}\sigma_N \lesssim \Bigg(\sum_{k\in\mathbb{Z}_{+}^{3}, |k|\le M} \frac{\|\theta \|_{L^1}^2}{\left\vert k\right\vert ^{2}} \Bigg)^{-1/2}= \|\theta \|_{L^1}^{-1} \Bigg(\sum_{k\in\mathbb{Z}_{+}^{3}, |k|\le M} \frac{1}{\left\vert k\right\vert ^{2}} \Bigg)^{-1/2}; $$
the last quantity vanishes as $M\to\infty$.
\end{proof}

\subsection{Remarks on collisions, turbulence and statistical mechanics}\label{subsect collisions}

A complicated open problem in the rigorous mathematical analysis of
interacting charged particles is whether a collision term should appear or not
in the kinetic description by a Vlasov type equation; and the consequences
that this fact may have on the statistical properties of particles. Let us
briefly discuss this topic, in order to illustrate the potential interest of
our result in this connection.

Charged particle interaction is long range, mediated by the electric and
magnetic fields, hence a mean field term must appear in the kinetic equation.
The difficult question is whether a collision term should also appear. Indeed,
the interaction kernel is singular at the origin, hence the pairwise
interaction of two very close particles is similar to a collision, or at least
more similar to a collision than to the infinitesimal interaction that
cumulates in a mean field way in the case of smooth long range interactions.
Therefore, some form of collision operator could appear, maybe depending on
the degree of sparseness of the system. In the regime considered in the
outstanding work of S. Serfaty \cite{Serfaty} (the first rigorous work
covering the true Coulomb potential), the collision operator does not arise,
the final equation is purely mean field; but this work adds the mean field
factor $\frac{1}{N}$ in front of the interactions, $N$ being the particle
number, factor that does not explicitly appear in a traditional physical model
of interacting ions and electrons.

When collisions exist, like in ordinary uncharged atoms or molecules of a
gas, in the macroscopic description of the particle density it is natural to
find a Laplacian describing particle diffusion. Each interaction, typically
very short range, produces a quite unpredictable (in the sense of very
unstable) change of direction of the motion of the particle. The accumulation
of these almost random changes leads to a Brownian behavior and to a diffusion
property at the density level. This is rigorously formalized in a number of
works, see in particular \cite{Bodine}.

For ions and electrons, we do not know whether ``Coulomb collisions'' (namely
the relatively strong change of direction of motion occurring when particles
are very close) may give rise to a true diffusive behavior.

Our result proves that a diffusive behavior occurs due to small scale
turbulence. We rigorously identify a set of assumptions under which one can
prove that diffusion (a Laplacian) arises in the phase density. It is
diffusion in velocity, namely a form of randomness in the velocity that may be
a surrogate of the result of collisions.

The diffusive behavior, in turn, may have consequences on statistical
mechanics properties. Uncharged particles usually satisfy a Maxwellian local
equilibrium, which quantifies the wide randomness of velocity directions of
particles, due in particular to collisions. In plasmas, the validity of the
Maxwellian distribution is controversial, since collisions are less efficient;
it is presumably more correct for electrons, possibly less for ions. The work
\cite{Ewart et al} and references quoted therein discuss this topic. The main
point of \cite{Ewart et al} is that turbulence acts as a substitute of
collisions, producing the randomness needed to establish a form of statistical
equilibrium, which however is not Maxwellian. We hope to clarify in future
works if the diffusion in velocity proved to exist in our work may have
consequences on the statistical equilibrium, like the one investigated in
\cite{Ewart et al}.

\bigskip

\noindent\textbf{Acknowledgments.} The research of the first author is funded by the European Union (ERC, NoisyFluid,
No. 101053472). The second author is grateful to the National Key R\&D Program
of China (No. 2020YFA0712700), the National Natural Science Foundation of
China (Nos. 12090010, 12090014). Views and opinions expressed are however those of
the authors only and do not necessarily reflect those of the European Union or
the European Research Council. Neither the European Union nor the granting
authority can be held responsible for them.

\end{document}